\documentclass[12pt]{amsart}
\usepackage{url, mathrsfs}
\usepackage{amssymb}
\usepackage{hyperref}
\usepackage{amsmath}
\usepackage{graphicx, epstopdf}
\usepackage{subfig}
\usepackage{color}
\usepackage{bbm}
\usepackage{subfloat}
\usepackage[draft]{fixme}
\usepackage{bm} 
\usepackage{bbm} 
\usepackage{epigraph}
\usepackage{hyperref}
\usepackage{endnotes}
\usepackage{ifpdf}
\usepackage{array}
\usepackage{wrapfig}
\usepackage{longtable}
\usepackage[framemethod=TikZ]{mdframed}
\usepackage{lipsum}
\usepackage{algorithm}
\usepackage[noend]{algpseudocode}
\usepackage{multicol}
\usepackage{multirow}
\usepackage{graphicx}
\usepackage{tikz}
\usepackage{tikz-cd}
\usepackage{makecell}

\newcommand{\Q}{\mathbb{Q}}
\newcommand{\N}{\mathbb{N}}
\newcommand{\Z}{\mathbb{Z}}
\newcommand{\R}{\mathbb{R}}
\newcommand{\F}{\mathbb{F}}
\newcommand{\K}{\mathbb{K}}
\newcommand{\C}{\mathbb{C}}

\theoremstyle{plain}

\newtheorem{thm}{Theorem}[section]
\newtheorem{prop}[thm]{Proposition}
\newtheorem{cor}[thm]{Corollary}
\newtheorem{lem}[thm]{Lemma}

\newtheorem{conjecture}[thm]{Conjecture}

\newtheorem*{claim*}{Claim}
\newtheorem*{thm*}{Theorem} 
\newtheorem*{cor*}{Corollary}

\newtheoremstyle{case}{}{}{}{}{}{:}{ }{}
\theoremstyle{case}

\theoremstyle{definition}

\newtheorem{exampleth}[thm]{Example}
\newenvironment{example}{\begin{exampleth}}{\hfill
    $\diamond$\\ \end{exampleth}}

\newtheorem{remark}[thm]{Remark}

\newtheorem{assumptions}[thm]{Assumptions}
\newtheorem{main}{Main Result}
\newtheorem*{exampleth*}{Example}
\newenvironment{example*}{\begin{exampleth*}}{\hfill
    $\diamond$ \end{exampleth*}}

\DeclareMathOperator{\SL}{SL}
\DeclareMathOperator{\diag}{diag}

\DeclareMathOperator{\HypDet}{HypDet}
\DeclareMathOperator{\Mat}{Mat}

\newcommand\ci{\mathbbm{i}}

\usepackage[inner=1in,outer=1in,top=1in,bottom=1in]{geometry}

\begin{document}
	\nocite{*}
	\title{Determinantal representations and the image of the principal minor map}
	
	\author{Abeer Al Ahmadieh}
	\address{Department of Mathematics, University of Washington, Seattle, WA, USA 98195} 
	\email{aka2222@uw.edu}

	\author{Cynthia Vinzant}
	\address{Department of Mathematics, University of Washington, Seattle, WA, USA 98195}
	\email{vinzant@uw.edu}

	\begin{abstract} In this paper we explore 
	determinantal representations of multiaffine polynomials 
	and consequences for the image of various spaces of matrices 
	under the principal minor map. 	
	We show that a real multiaffine polynomial has a definite Hermitian 
	determinantal representation if and only if all of its so-called Rayleigh differences 
	factor as Hermitian squares and use this characterization to conclude that 
	the image of the space of Hermitian matrices under the principal 
	minor map is cut out by the orbit of finitely many equations and inequalities 	
	under the action of $({\rm SL}_2(\mathbb{R}))^{n} \rtimes S_{n}$. 
	We also study such representations over more general fields with quadratic extensions. 
	Factorizations of Rayleigh differences prove an effective tool for capturing subtle behavior of the principal minor map. 
	In contrast to the Hermitian case, we give examples to show for any field $\mathbb{F}$, there is no 
	finite set of equations whose orbit under $({\rm SL}_2(\mathbb{F}))^{n} \rtimes S_{n}$
	cuts out the image of $n\times n$ matrices over $\F$ under the principal minor map for 
	every $n$. 
	\end{abstract}

	\maketitle

	\section{Introduction}\label{sec:Intro}
	Given an $n \times n$ matrix $A$ with entries in a field $\F$, let $A_S$ denote the determinant of the submatrix of $A$ indexed by the set $S$ on the rows and columns. If we set $A_\emptyset = 1$, the principal minors of a matrix form a vector of length $2^n$. The \textit{principal minor map} is the map that assigns to each matrix the vector of its principal minors, namely
	\[
	\varphi: \F^{n\times n} \longrightarrow \F^{2^n} \ \ \text{ given by } \ \ A \longrightarrow \left(A_S\right)_{S\subseteq[n]}.
	\]	
	One of the motivating goals of this paper is to characterize the image of this map. 
	This problem dates back to the $19$th century \cite{muir1900x}, \cite{nanson1897xlvi}. In the cases $n=2$ and $n=3$, this image is characterized by $A_\emptyset = 1$ over $\C$. In the case $n=4$, Lin and Sturmfels \cite{LinSturmfels09} give an explicit list of $65$ polynomials that cutout the image and they conjectured that it is cutout by equations of degree $12$ for any $n$. 
	
	The image of the space of real and complex symmetric matrices was studied by Holtz and Sturmfels \cite{HoltzSturmfels07}, who show that the image is closed and invariant under an action of the group $\SL_2(\mathbb{R})^n \rtimes S_n$ and conjectured that the vanishing of polynomials in the orbit of the \emph{hyperdeterminant} under this group cuts out the image of the principal minor map over $\C$. This conjecture was resolved by Oeding \cite{Oeding11}. In \cite{AV21}, we build of techniques in 	\cite{KPV15} to generalize this result to hold over arbitrary unique factorization domain. Here we use similar techniques to characterize the image of Hermitian matrices.

	The principal minor map problem appears in many  different fields and applications, including statistical models, machine learning, combinatorics and matrix theory. One fundamental application is the study of determinantal point processes (DPP). These are probabilistic models that arise naturally in the study of random matrix theory \cite{johansson2005random} and machine learning \cite{chao2015large, gartrell2016bayesian}. 
	Symmetric DPPs have attracted a lot of attention as they reflect the repulsive behavior in modeling, see \cite{affandi2014learning,borodin2009determinantal, dupuy2018learning, kulesza2012learning,urschel2017learning}. Non-symmetric kernels are also of interest for modeling both repulsive and attractive interactions \cite{alimohammadi2021fractionally,brunel2018learning,gartrell2019learning}.
Learning the parameters of such a model from data leads to the computation problem of reconstructing 
a matrix from the vectors of its principal minors. Griffin and Tsatsomeros \cite{GriffinTsatsomeros06,GriffinTsatsomeros06s} give a numerical algorithm that reconstructs a preimage of a matrix, if it exists over $\C$. Rising, Kulesza and Taskarc \cite{RisingKuleszaTaskar15} provide an efficient algorithm for reconstruction in the symmetric case.

	In this paper, we study the principal minor map via determinantal representations of an associated 
	multivariate polynomial. Explicitly, 
	To each vector ${\bf a} = \left(a_{S}\right)_{S\subset [n]} \in \F^{2^n}$, we assign a multiaffine polynomial $f_{\bf a}$ where $f_{\bf a} = \sum_{S\subset [n]} a_S \prod_{i\in [n]\setminus S} x_i$. This transforms the problem of characterizing the image of the principal minor map to the problem of characterizing multiaffine polynomials with determinantal representation, these are polynomials that can be written in the following form: $f = \det\left(\diag(x_1,\hdots,x_n)+A\right)$ for some $n\times n$ matrix $A$. Symmetric (Hermitian) multiaffine determinantal polynomials are determinantal polynomials that corresponds to symmetric (Hermitian) matrices. In \cite{AV21} we prove that the class of symmetric determinantal multiaffine polynomials is characterized by their Rayleigh differences being squares.
	
	The Rayleigh difference of a polynomial $f$ with respect to $i,j\in [n]$ is defined to be 
	\begin{equation}\label{eq:DeltaDef}
		\Delta_{ij}(f)  = \frac{\partial f}{\partial x_i}  \frac{\partial f}{\partial x_j} - f \frac{\partial^2 f}{\partial x_i \partial x_j}.
	\end{equation}
	
	Here we also use them to characterize Hermitian determinantal multiaffine polynomials over any field $\K$ with an automorphism of order two and deduce a characterization of Hermitian determinantal multiaffine polynomials over $\C$.
	
	\begin{main}[Corollary~\ref{cor:HerCompDetRep}, Theorem~\ref{thm:GenDet}]
	A real multiaffine polynomial $f$ has a linear Hermitian detertminantal representation 
	if and only if all of its Rayleigh differences $\Delta_{ij}(f)$ factor as Hermitian squares. 
	\end{main}
	
	One of the themes of this paper is that factorizations of Rayleigh differences can capture 
	subtle behavior of the principal minor map. 
	
	\begin{example*}[Example~\ref{ex:HerEx}]
	For example the Rayleigh differences of the polynomial 	
	\[
	f_{\bf a}(x_1,x_2,x_3,x_4) = x_1 x_2 x_3 x_4 - x_1 x_2 -  x_1 x_3 -  x_1 x_4 - x_2 x_3 - x_2 x_4 -  x_3 x_4 +1
	\]
factor into Hermitian squares in multiple ways, e.g. $\Delta_{34}(f_{\bf a}) = (x_1 - \ci)(x_1 + \ci) (x_2 - \ci)(x_2+\ci)$.
These different choices of factorization capture some non-generic behavior in the fiber of the principal minor map and lead to three different determinantal representations of $f$: 
	\[
	\left\{
	\small\begin{pmatrix}
		x_1 & -1 & \ci & \ci \\ -1 & x_2 & -1 & -1 \\ -\ci & -1 & x_3 & -\ci \\ -\ci & -1 & \ci & x_4
	\end{pmatrix}, \ \ 
	\small\begin{pmatrix}
		x_1 & -1 & \ci & \ci \\ -1 & x_2 & -1 & -1 \\ -\ci & -1 & x_3 & \ci \\ -\ci & -1 & -\ci & x_4
	\end{pmatrix}, \ \ 
	\small\begin{pmatrix}
		x_1 & -1 & \ci & -\ci \\ -1 & x_2 & -1 & -1 \\ -\ci & -1 & x_3 & -\ci \\ \ci & -1 & \ci & x_4
	\end{pmatrix}
	\right\}.
	\]	
	\end{example*}
	In general, the fibers of the principal minor map are not well understood.
	In the symmetric case, the fibers were characterized by Engel and Schneider \cite{engel1980matrices}.  In $1984$, Loewy and Hartfiel \cite{hartfiel1984matrices} and then Loewy \cite{loewy1986principal} gave sufficient conditions for two general matrices to be \textit{diagonally similar} and hence to belong to the same fiber, but as the example above shows, the fiber in general can be larger. In future work, we hope to use the techniques developed in this paper to give a better understanding of the fibers of the principal minor map.  
	
	Here we use the classical theory of determinantal representations to understand the principal minor map, including ideas from Dixon \cite{dixon} on the construction of symmetric determinantal representations of plane curves. 
	The study of symmetric and Hermitian determinantal representations is also closely 
	related to the theory of hyperbolic and real stable polynomials, which are multivariable  
	generalizations of real-rooted univariate polynomials. 
	Since then, hyperbolic and stable polynomials have found wide-spread applications in 
combinatorics \cite{GurvitsPerm, RamGraph}, convex analysis \cite{ConvAnalysis}, operator theory \cite{HeltonVinnikov07,kadSing}, 
probability \cite{LiggetNegDep}, and theoretical computer science \cite{AOR, LGS_TCS, SV_TCS}. 
	The question of which stable polynomials have definite Hermitian determinantal representations has implications in operator theory and the theory of convex optimization. See \cite{VinnikovSurvey} for more. 
In general, the existence of definite Hermitian representations does not follow from the existence of general representations over $\C$.   In this paper we show that this is not the case for multiaffine polynomials. 

	\begin{main}[Theorem~\ref{thm:DetStableHerm}]
 If a multiaffine real stable polynomial $f$ has a linear determinantal representation over $\C$, then it has a definite Hermitian determinantal representation.
	\end{main}

From the classification above, we characterize the image of Hermitian matrices under the principal minor
map by characterizing the set of real multiquadratic polynomials that factor as Hermitian squares.
This leads to explicit equations and inequalities defining the image.  

	\begin{main}[Corollary~\ref{cor:ImgHer}]
	The image of the set of $n\times n$ Hermitian matrices under the principal minor map 
	is cut out by the orbit under $\SL_2(\R)^n \rtimes S_n$ 
	of two inequalities and three degree-12 equations defined by polynomials in 
	$ \Q[a_S:S\subseteq 4]$. 
\end{main}

An explicit description of the image of general $n\times n$ matrices remains mysterious. Huang and Oeding \cite{HuangOeding17} give description of the image in the special case where all principal minors of same size are equal (\textit{the symmetrized principal minor assignment problem}) where they use the cycle sums in their approach. They provide a minimal parametrization of the respective varieties in the cases of symmetric, skew symmetric and square complex matrices. Kenyon and Pemantle \cite{KenyonPemantle14} adjust the principal minor map by adding \textit{the almost principal minors} to the vector in the image and they showed that the ideal of the variety in this case is generated by translations of a single relation using the \text{rhombus tiling}. 
	
	Using factorizations of Rayleigh differences, we found a family of examples 
	that shows that for general $n\times n$ matrices, such a finite description is impossible. 	

	\begin{main}[Theorem~\ref{thm:GenCase}]
		For any field $\F$, there is \emph{no finite set of equations} 
		whose orbit under ${\rm SL}_2(\F)^n\rtimes S_n$ 
		cuts out the image of the principal minor map for all $n$.
	\end{main}
	In the case $n=5$ of instance, the polynomial 
	\begin{equation}\label{eq:f5}
		f = x_1 (x_3 x_4 + 1)(x_2 x_5+1) + (x_2 x_3 + 1)(x_4 x_5 + 1)
	\end{equation}
	is not determinantal, i.e. its vector of coefficients do not belong to the image of the principal minor map, but it is determinantal after specializing any one variable.

	This paper is organized as follows. 
	In Section~\ref{sec:Background}, we introduce terminology and the basic properties of 
	determinantal representations and the action of ${\rm SL}_2(\F)^n\rtimes S_n$. 
	In Section~\ref{sec:DetRep}, we give a characterization of multiaffine determinantal polynomials 
	involving the factoring of Rayleigh differences. 
	For Hermitian determinantal representations, this condition simplifies and we give an 
	algorithm for constructing such representations from a factorization, 	
	as described in Section~\ref{sec:MA_algebra} and Section~\ref{sec:otherDetRep}. 
	In Section~\ref{sec:StableDetRep} 
	we give a characterization of multiaffine stable determinantal polynomials and 
	prove Theorem~\ref{thm:DetStableHerm}.
	In Section~\ref{sec:factoringMQ},
	we translate these conditions into explicit equations and inequalities whose orbit 
	under ${\rm SL}_2(\R)^n\rtimes S_n$ cuts of the image of Hermitian matrices 
	under the principal minor map.  Finally, in Section~\ref{sec:counterex}, we conclude 
	by presenting a family of examples that disproves the existing of such a finite description 
	for the image of general $n\times n$ matrices under the principal minor map.

{\bf Acknowledgements}. We would like to thank Mario Kummer and Bernd Sturmfels for 
the helpful comments and discussions.
Both authors were partially supported by the National Science Foundation Grant DMS-1943363.

\section{Background and notation}\label{sec:Background}
	
	For a commutative ring $R$, we use $R[{\bf x}]$ to denote the polynomial ring $R[x_1, \hdots, x_n]$ and for $f\in R[{\bf x}]$, 
	we use $\deg_i(f)$ to denote the degree of $f$ in the variable $x_i$. 
	For ${\bf d} = ({\bf d}_1, \hdots, {\bf d}_n)$ with $d_i \in \Z_{\geq 0}$, let $\F[{\bf x}]_{\leq {\bf d}}$ denote the set of 
	polynomials with degree at most ${\bf d}_i$ in $x_i$ for each $i=1, \hdots, n$. These form an $R$-module  
	of rank $\prod_{i=1}^n ({\bf d}_i+1)$. 
	When ${\bf d}_1 = \hdots = {\bf d}_n = m$, we abbreviate $R[{\bf x}]_{\leq (m,\hdots m)}$ by 
	$R[{\bf x}]_{\leq {\bf m}}$.  
	Of particular interest are \emph{multiaffine polynomials}, with degree $\leq 1$ in each variable, and \emph{multiquadratic polynomials}, with degree $\leq 2$ in each variable. 
	These are denoted by  $R[{\bf x}]_{\leq {\bf 1}} = R[{\bf x}]_{\rm MA}$ and $R[{\bf x}]_{\leq {\bf 2}} = R[{\bf x}]_{\rm MQ}$, respectively. 
	
	We use ${\rm Mat}_n(\F)$ to denote the set of $n\times n$ matrices with entries in $\F$. 
	When $\K$ is a field with an automorphic involution $a\mapsto \overline{a}$, 
	we use ${\rm Her}_n(\K)$ to denote the set of matrices $A\in  {\rm Mat}_n(\K)$ for which $\overline{A} = A^T$. Note that for $\K = \C$ and  $a\mapsto \overline{a}$ given by complex conjugation, this is the usual set of $n\times n$ Hermitian matrices.

\subsection{The action of ${\rm SL_2}(R)^n\rtimes S_n$ and homogenezations} The action of $\SL_2(R)^n$ on $R[{\bf x}]_{\leq {\bf d}}$ is defined as follows. 
	Let $\gamma =(\gamma_i)_{i\in [n]} $ in $\SL_2(R)^{n}$ where $\gamma_i = \tiny{\begin{pmatrix} a_i& b_i  \\ c_i & d_i \end{pmatrix}}$. Then for $f \in R[{\bf x}]_{\leq {\bf d}}$, 
	\[
	\gamma\cdot f = \prod_{i=1}^n (c_i x_i + d_i)^{{\bf d}_i} \cdot f\left(\frac{a_1 x_1 + b_1}{c_1 x_1 + d_1}, \hdots, \frac{a_n x_n + b_n}{c_n x_n + d_n}\right).
	\]
	One way to interpret this action is with the multihomogenezation of $f$. Let 
	$f^{\rm {\bf d}-hom}$ in $ R[x_1, \hdots, x_n, y_1, \hdots, y_n]_{\bf d}$ denote the polynomial 
	\[
	f^{\rm {\bf d}-hom} = \prod_{i=1}^n y_i^{{\bf d}_i} \cdot f\left(x_1/y_1, \hdots, x_n/y_n\right). 
	\]
	The induced action of $\gamma$ on $f^{\rm {\bf d}-hom} $ is just a linear change of coordinates: 
	\[\gamma \cdot f^{\rm {\bf d}-hom} = f^{\rm {\bf d}-hom}\left(\gamma_1\cdot\begin{pmatrix} x_1 \\ y_1\end{pmatrix}, \hdots,  \gamma_n\cdot\begin{pmatrix} x_n \\ y_n\end{pmatrix}\right).\]
	Restricting to $y_1 = \hdots y_n = 1$ gives back $\gamma\cdot f $.

We will also use the usual homogenization of a polynomial to some total degree $d$, using a single homogenizing variable $y$. 
That is, for $f = \sum_{\alpha}c_{\alpha}{\bf x}^{\alpha}\in R[{\bf x}]$ of total degree $d=\deg(f)$, 
its homogenization is 
\[
f^{\rm hom} = y^d f\left(x_1/y, \hdots, x_n/y\right)=
 \sum_{\alpha}c_{\alpha}{\bf x}^{\alpha}y^{d - |\alpha|}  \ \in \ R[{\bf x}, y].
 \]

Suppose that $\K$ is a field with an automorphic involution $a\mapsto \overline{a}$ with fixed field $\F$. 
This extends to an involution on $\K[{\bf x}]$ by acting on the coefficients. 
We will say that a polynomial $q\in \F[{\bf x}]$ is a \emph{Hermitian square} 
if $q = g\overline{g}$ for some $g\in \K[{\bf x}]$. 
To end this section, we remark that for $f\in \F[{\bf x}]$, 
the condition that $\Delta_{ij}(f)$ 
is a Hermitian square is robust to homogenization. 

\begin{prop}\label{prop:homDelta}
Suppose that $\K$ is a field with an automorphic involution $a\mapsto \overline{a}$ with fixed field $\F$. 
Let $f\in \F[{\bf x}]$. 	For $i,j\in [n]$, the polynomial  
	$\Delta_{ij}(f)$ is a  Hermitian square if any only if 
	$\Delta_{ij}(f^{\rm hom})$ is a  Hermitian square. 
\end{prop}
\begin{proof}
If $\Delta_{ij}(f^{\rm hom})$ is a Hermitian square, then specializing to $y=1$ 
gives a representation of $\Delta_{ij}(f)$ as a Hermitian square. 
For the converse, 
let $f\in \F[{\bf x}]$ with total degree $d$ and suppose that 
$\Delta_{ij}f = g\overline{g}$ for some $g \in \K[{\bf x}]$. Let $m =\deg(g)= \deg(\overline{g})$. 
By definition, $\Delta_{ij}(f^{\rm hom}) \in \F[{\bf x}, y]$ 
is homogeneous of degree $2d-2$. Its restriction to $y=1$ equals $\Delta_{ij}f$. 
Therefore $\Delta_{ij}(f^{\rm hom})$ equals $y^{2d-2-2m}(\Delta_{ij}(f))^{\rm hom}$, 
which is the Hermitian square $h\overline{h}$ where $h$ is the homogenezation of $g$ 
to total degree $2d-1$. 
\end{proof}

\subsection{The action of $\SL_2(\F)$ on matrices}
Given a matrix $A\in \Mat_n(\F)$, consider the multiaffine polynomial $f = \det\left(\diag\left(x_1,\hdots,x_n\right)+ A\right)$. 
For $\gamma =(\gamma_i)_{i\in [n]} $ in $\SL_2(\F)^{n}$ with $\gamma_i = \tiny{\begin{pmatrix} a_i& b_i  \\ c_i & d_i \end{pmatrix}}$, $\gamma \cdot f$ is defined by:
\[
\gamma\cdot f = \prod_{i=1}^n (c_i x_i + d_i) \cdot \det\left(\diag\left(\frac{a_1 x_1 + b_1}{c_1 x_1 + d_1}, \hdots, \frac{a_n x_n + b_n}{c_n x_n + d_n}\right)+ A\right).
\]
Let $A_i$ denote the $i$th column of $A$ and $e_i$ the vector whose $i$th entry is one and zero otherwise. 
By using the factor $ (c_i x_i + d_i) $ to scale the $i$th column, we see that 
\[
\gamma \cdot f = \det\left(C\diag(x_1, \hdots, x_n)+ B\right)
\]
where $C$ is the matrix with $i$th column $C_i = (a_ie_i + c_iA_i)$ and $B$ is the matrix with $i$th column $B_i = b_ie_i + d_i A_i$.  
When the matrix $C$ is invertible, this gives 
\[
\gamma\cdot f = \det(C)\det\left(\diag\left(x_1, \hdots, x_n\right)+ C^{-1}B\right).
\]
Up to the scalar multiple $\det(C)$, the coefficients of $\gamma\cdot f$ are the principal minors of the matrix $C^{-1}B$.

\subsection{Resultants}	
For two univariate polynomials $a = \sum_{j=0}^d a_j t^j$ with $a_d\neq 0$ and $b = b_1t + b_0$ with $b_1\neq 0$ we define the resultant of $a$, $b$ with respect to the variable $t$ to be
	\[
	{\rm Res}_{t}(a,b) = \sum_{j=0}^d a_j (-b_0)^j(b_1)^{d-j}.
	\]	
	Over an algebraically closed field, this polynomial vanishes if and only if the univariate polynomials $a$ and $b$ have a common root. 
See, for example, \cite[\S3.5]{CLO}.  
We will focus on multiaffine polynomials and so focus on resultants in degree $d=1$. 
For $k=1, \hdots, n$, define
\[
{\rm res}_{x_k}(g,h) = (g|_{x_k=0}) \cdot \frac{\partial}{\partial x_k} h - (h|_{x_k=0}) \cdot \frac{\partial}{\partial x_k}g. 
\]
In particular, if $g$ and $h$ both have degree one in $x_k$, this agrees with ${\rm Res}_{x_k}(g,h)$. The benefit of this degree-dependent definition is that it is invariant under the action of $\SL_2(R)$. 

If $f\in R[{\bf x}]$ has degree $\leq 1$ in both $x_i$ and $x_j$, then
\begin{equation}\label{eq:DeltaRes}
\Delta_{ij}(f) = {\rm res}_{x_i}\left(\frac{\partial f}{\partial x_j}, f|_{x_j=0} \right)
= {\rm res}_{x_j}\left(\frac{\partial f}{\partial x_i}, f|_{x_i=0} \right).
\end{equation}

\begin{prop}\label{prop:DeltaZero}
If $f\in R[x_1, \hdots, x_n]$ has degree one in each of $x_i$ and $x_j$, 
then $\Delta_{ij}(f)=0$ if and only if $f$ factors into polynomial $g\cdot h$ with 
$g\in R[x_k : k\neq i]$ and $h\in R[x_k:k\neq j]$.
\end{prop}
\begin{proof}
By assumption we can write $f = ax_ix_j + b x_i + c x_j + d$ for $a,b,c,d\in R[x_k: k\neq i,j]$. 
Then $\Delta_{ij}(f) = bc-ad$. If $\Delta_{ij}(f) = 0$, then there is some factorization 
$b = b_1 b_2$ and $c = c_1c_2$ for which $a = b_1c_1$ and $d = b_2c_2$.  
Then $f = (b_1 x_i + c_2)(c_1x_j + b_2)$.  Similarly, if $f = (b_1 x_i + c_2)(c_1x_j + b_2)$ for some $b_1, b_2, c_1, c_2\in R[x_k: k\neq i,j]$, then $\Delta_{ij}(f) = bc-ad=0$. 
\end{proof}

\begin{prop}\label{prop:SL2_on_res}
Let $g\in R[{\bf x}]_{\leq {\bf d}}$ and $h\in R[{\bf x}]_{\leq {\bf e}}$ 
with ${\bf d}_k = {\bf e}_k = 1$. 
For $\gamma\in {\rm SL}_2(R)^n$,
\[
 \gamma \cdot {\rm res}_{k}(g, h) 
=  {\rm res}_{k}(\gamma \cdot g, \gamma \cdot h),
\]
where $\gamma$ acts of on ${\rm res}_{k}(g, h)$ as 
polynomial of multidegree $\leq  {\bf d} + {\bf e} - 2\cdot {\bf 1}_k$ with ${\bf 1}_k$ is the vector with 
$k$th entry is $1$ and zero otherwise. 
\end{prop}
\begin{proof}
Write $g = g_1x_k +g_0$ and $h = h_1x_k + h_0$ where $g_1, g_0, h_1,h_0$
are polynomials in the polynomial ring $R[x_j: j\neq k]$. 
The resultant ${\rm res}_{k}(g, h)$ is the determinant of the $2\times 2$
matrix ${\small \begin{pmatrix} 
h_1 & h_0 \\ 
g_1 & g_0 
\end{pmatrix}}$. 
Consider $\gamma = {\small \begin{pmatrix} a & b \\ c & d \end{pmatrix}}\in \SL_2(R)$ acting on the $j$th coordinate. 
If $j=k$, then 
\[
\gamma\cdot g =  g_1(ax_k+b) +g_0(c x_k+d), \ \ \text{ and } \ \ 
\gamma\cdot h =  h_1(ax_k+b) +h_0(c x_k+d).
\ 
\]
Taking coefficients with respect to $\{1,x_k\}$, we see that 
the ${\rm res}_{x_k}(\gamma\cdot g, \gamma \cdot h)$
equals
\[\det\begin{pmatrix}
a h_1 + c h_0  & b h_1 + d h_0\\
a g_1 + c g_0 & b g_1 + d g_0 
\end{pmatrix}
= 
\det\left(\begin{pmatrix} 
h_1 & h_0 \\ 
g_1 & g_0 
\end{pmatrix}
\begin{pmatrix} 
a & b \\ 
c & d 
\end{pmatrix}\right)= 
\det\begin{pmatrix} 
h_1 & h_0 \\ 
g_1 & g_0 
\end{pmatrix} = {\rm res}_{x_k}( g,  h).
\]
Since $\gamma$ acts on $R[{\bf x}]_{\leq {\bf d} + {\bf e} - 2\cdot {\bf 1}_k}$ as the identity, 
this equals $\gamma\cdot {\rm res}_{x_k}( g,  h)$. 

If $j\neq k$, then 
$\gamma\cdot g = (\gamma\cdot g_1)x_k + (\gamma \cdot g_0)$ 
and $\gamma\cdot h = (\gamma\cdot h_1)x_k + (\gamma \cdot h_0)$, 
where $\gamma$ acts on $g_1,g_0$ and $h_1, h_0$ as elements of 
multidegree ${\bf d}-{\bf 1}_k$ and ${\bf e}-{\bf 1}_k$, respectively.
It follows that 
\[
 {\rm res}_{x_k}(\gamma\cdot g, \gamma \cdot h)
=\det\begin{pmatrix} 
\gamma\cdot h_1 & \gamma\cdot h_0 \\ 
\gamma\cdot g_1 & \gamma \cdot g_0 
\end{pmatrix} =\gamma\cdot {\rm res}_{x_k}( g,  h).
\]
\end{proof}

From \eqref{eq:DeltaRes}, this gives the following: 
\begin{cor}\label{cor:SL2onDelta}
Consider an element $\gamma \in \SL_2(R)^{n}$ that acts by 
$\tiny{\begin{pmatrix} a&  b \\ c & d \end{pmatrix}}$ in the $k$-th coordinate 
and the identity in all others.  For any $f\in R[{\bf x}]_{\leq {\bf 1}}$, 
\[
\Delta_{ij}(\gamma \cdot f) = \begin{cases} 
 \Delta_{ij}(f) & \text{ if } k = i, j \\
\gamma \cdot \Delta_{ij}(f) & \text{ otherwise.}
\end{cases}
 \]
 \end{cor}

\section{Determinantal Representations and Rayleigh Differences}\label{sec:DetRep}

Let $R$ be a unique factorization domain and denote by ${\rm Mat}_n(R)$ the set of $n\times n$ matrices with entries in $R$.

\begin{thm}\label{thm:DeltaDetRep}
Let $f \in R[x_1,\hdots,x_n]$ be
multiaffine in the variables $x_1,\hdots,x_n$ with its coefficient of $x_1\cdots x_n$ equals one.
Then $f= \det({\rm diag}(x_1,\hdots,x_n) +A)$ for some $A\in {\rm Mat}_n(R)$ if and only if for every $i\neq j\in [n]$, the polynomials $\Delta_{ij}(f)$ factor as the product $g_{ij}\cdot g_{ji}$ where 
\begin{itemize}
\item[(a)] $g_{ij}\in  R[x_k : k\neq i,j]$ is multiaffine in $x_1,\hdots, x_n$ and  
\item[(b)] for every $k\in [n]\backslash \{i,j\}$, ${\rm res}_{x_k}(g_{ij}, f) = g_{ik}g_{kj}$. 
\end{itemize}
In this case, we can take $g_{ij}$ to be the $(i,j)$th entry of $({\rm diag}(x_1,\hdots,x_n) +A)^{\rm adj}$, with $M^{\rm adj}$ represents the adjugate matrix of $M$. 
%Moreover, for $R = \C$, we can take the matrices $A$ to be Hermitian if and only if $f\in \R[x_1, \hdots, x_n]$ and $g_{ij} = \overline{g_{ji}}$ for all $i,j \in [n]$. 
\end{thm}

\begin{proof}[Proof of ($\Rightarrow$)]
This follows from a classical equality on the principal minors of an $n\times n$ matrix, used by 
Dodgson \cite{Dodgson} as a method for computing determinants.  
This is also known as the Desnanot-Jacobi identity or more generally as Sylvester's determinantal identity. 
For subsets $S, T\subset [n]$ of equal cardinality, 
let $M(S,T)$ denote the submatrix of $M$ obtained by \emph{dropping} rows $S$ and columns $T$ from $M$. 
Then for any $i\neq j\in [n]$, 
\begin{equation}\label{eq:dodgson} 
\det(M(i,k))\cdot \det(M(j,\ell))  - \det(M)\cdot \det(M(\{i,j\}, \{k,\ell\})) = \det(M(i,\ell))\cdot \det(M(j,k)).
\end{equation}
Note that for $M = {\rm diag}(x_1,\hdots,x_n) + A$ and any subset $S\subseteq [n]$, 
the principal minor $\det(M(S,S))$ equals the derivative of $f$ with respect to the variables in $S$,  
$\left(\prod_{i\in S}\frac{\partial}{\partial x_i}\right) f$. 
The equation above with $k=i$ and $\ell = j$ 
then gives that $\Delta_{ij}(f)$ equals $\det(M(i,j))\cdot \det(M(j,i))$. 

For every $i,j\in [n]$, let $g_{ij}$ denote $\det(M(i,j))$. Then $g_{ij}\in R[x_k: k\neq i,j]$ is multiaffine in $x_1, \hdots, x_n$ and $\Delta_{ij}(f) = g_{ij}g_{ji}$. 
Under an appropriate 
choice of indices, \eqref{eq:dodgson} gives
\[g_{kk} \cdot g_{ij} - f \cdot q =  g_{ik}\cdot g_{kj}  \ \ \text{ where } \ \ q =\det(M(\{i,k\}, \{j,k\})).\]
Note that $g_{kk} = \frac{\partial f}{\partial x_k}$ is the coefficient of $x_k$ in $f$ 
and $q$ is the coefficient of $x_k$ in $g_{ij}$. Therefore $g_{kk} \cdot g_{ij} - f \cdot q $ is 
the resultant of $g_{ij}$ and $f$ with respect to $x_k$. 
%Finally, we note that if $R = \C$ and the matrices $M_{n+1}, \hdots, M_m, M_0$ are Hermitian, then
%$f\in \R[x_1, \hdots, x_m]$ and  $g_{ij} = \overline{g_{ji}}$ for every $i,j$. 
\end{proof}

\begin{example}\label{ex:CounterExFactoring}
 For $n\geq 5$, one cannot remove condition (b) from Theorem~\ref{thm:DeltaDetRep}. 
Consider 
\begin{align*}
f  = \ & x_1 x_2 x_3 x_4 x_5 + x_1 x_2 x_3 x_4 + x_1 x_2 x_3 x_5 + x_1 x_2 x_4 x_5 + 
 x_1 x_3 x_4 x_5 + x_2 x_3 x_4 x_5\\
&  + x_1 x_2 x_4 +   x_1 x_2 x_5  + x_1 x_3 x_4 + 
 x_2 x_3 x_5 + x_3 x_4 x_5 .
\end{align*}
One can check that for every $i,j\in [5]$, $\Delta_{ij}(f)$ factors as the product of two multiaffine polynomials  
in $\Q[x_1, \hdots, x_5]$. For example, 
$\Delta_{12}(f) = -x_3 x_4 x_5 (x_4 x_5 -x_3 + x_4 + x_5 )$.
Since there in an irreducible factor involving all three variables, 
there is only one possible factorization of $\Delta_{12}(f)$ as the product of two multiaffine polynomials $g_{12}\cdot g_{21}$, up to scalar multiples and switching the factors, namely
 $g_{12} = -x_3 x_4 x_5$ and $g_{21} = x_4 x_5 -x_3 + x_4 + x_5$. 
 Taking the resultant of $g_{21}$ and $f$ with respect to $x_3$ gives 
 \[
{\rm Res}_{x_3}(g_{21}, f) = (x_1 x_5+x_1 + x_5 ) (x_2 x_4+x_2 + x_4) (x_4 x_5+x_4 + x_5 ). 
 \]
Each of the three quadratic factors are irreducible and so there is no way of writing 
this resultant as the product of \emph{two} multiaffine polynomials.  Therefore there is 
no choice of polynomials $g_{23}$ and $g_{31}$ satisfying the conditions in Theorem~\ref{thm:DeltaDetRep}. 
\end{example}

\begin{lem}\label{lem:factorDelta}
Let $f \in R[x_1, \hdots, x_n]$ be multiaffine in the variables $x_1,\ldots,x_n$ and its coefficient of $x_1\cdots x_n$ equals one.
If $f = g\cdot h$ for some $g,h\in R[x_1, \hdots, x_n]$, then $g$ and $h$ are multiaffine in 
disjoint subsets of the variables $x_1, \hdots, x_n$ and we can take their leading coefficients in these variables to be  one. 
Moreover, if the polynomials $\Delta_{ij}(f)$ have factorizations satisfying conditions $(a)$
and $(b)$ in Theorem~\ref{thm:DeltaDetRep}, then so do $\Delta_{ij}(g)$ and $\Delta_{ij}(h)$.
\end{lem}

\begin{proof}
For any $i\in [n]$, the degree of $f$ in $x_i$  must be the sum of the degrees of $g$ and $h$ in $x_i$. Since this sum of nonnegative numbers is one for each $i\in [n]$, 
we see that for some subset $I\subseteq[n]$, 
$g$ is multiaffine in $\{x_i : i\in I\}$, $h$ is multiaffine in $\{x_j : j\in [n]\backslash I\}$,
and $\deg_i(h) = \deg_j(g) = 0$ for any $i\in I$ and $j\not\in I$. 

The highest degree term in $f$ with respect to the variables $x_1, \hdots, x_n$, 
$\prod_{i=1}^nx_i$, is the product of the highest degree terms in $g$ and $h$. 
Therefore after rescaling, we can assume that both $g$ and $h$ have leading coefficient 
in these variables equal to $1$. 
For $i\in I$ and $j \notin I$, $\partial(g\cdot h)/\partial x_i = h\cdot \partial g/\partial x_i$ 
and $\partial (g\cdot h)/\partial x_j = g\cdot \partial h/\partial x_j$. 
From this, one can check that 
$\Delta_{ij}(gh)$ equals  $h^2\Delta_{ij}(g)$ for $i,j \in I$, 
$g^2\Delta_{ij}(h) $ for $i,j\in [n]\backslash I$ and zero otherwise. 

Suppose that for $i,j\in [n]$, $\Delta_{ij}(f) = m_{ij}m_{ji}$ with $m_{ij}$ multiaffine in $x_1, \hdots, x_n$ and 
${\rm res}_{x_k}(m_{ij}, f) = m_{ik}m_{kj}$ for every $i,j,k$. 
For $i,j\in I$, we see that $m_{ij}m_{ji} = h^2\Delta_{ij}(g)$.  Since $m_{ij}, m_{ji}$ are multiaffine, 
they both must be divisible by $h$, leaving $\tilde{m}_{ij}\tilde{m}_{ji} = \Delta_{ij}(g)$, where 
$\tilde{m}_{ij}, \tilde{m}_{ji}$ are multiaffine in $x_i$ for $i\in I$. 
Moreover, for $k$ also in $I$, 
\[
h^2 \tilde{m}_{ik}\tilde{m}_{kj} = 
m_{ik}m_{kj} = {\rm Res}_{x_k}(m_{ij}, f) = {\rm res}_{x_k}(\tilde{m}_{ij}h, gh) =  h^2{\rm res}_{x_k}(\tilde{m}_{ij}, g)
\]
showing that $\tilde{m}_{ik}\tilde{m}_{kj} = {\rm res}_{x_k}(\tilde{m}_{ij}, g)$.
The desired factorization for $\Delta_{ij}(h)$ with $i,j\in [n]\backslash I$ follows similarly. 
\end{proof}

\begin{proof}[Proof of ($\Leftarrow$)] 
Suppose that $f$ is irreducible and homogeneous of degree $n$. 
Let $G$ denote the $n\times n$ matrix with  $(i,j)$th entry $g_{ij}$ for $i\neq j$ and $g_{ii}:=\frac{\partial f}{\partial x_i}$ for $i=j$. 

We claim that all of the $2\times 2$ minors of $G$ lie in $\langle f \rangle$. 
This is immediate for the symmetric minors, as $g_{ii}g_{jj} - g_{ij}g_{ji} = f \cdot \frac{\partial^2 f}{\partial x_i \partial x_j}$.
Moreover, since $\frac{\partial f}{\partial x_1}$ is the coefficient of $x_1$ in $f$, 
the resultant ${\rm res}_{x_1}(g_{ij},f)$ has the form $\frac{\partial f}{\partial x_1} g_{ij} - q f$ for some $q$. 
This gives $g_{11}g_{ij} - g_{i1}g_{1j} = qf$. Finally, suppose that $i,j,k,\ell$ are all distinct. 
Then 
	\[
	g_{11}^2 (g_{ij}g_{kl} - g_{il}g_{kj})  
	= 
	(g_{11} g_{ij})(g_{11} g_{kl}) - (g_{11}g_{il})(g_{11}g_{kj})
	\equiv 
	g_{1i}g_{1j}g_{1k}g_{1l} -  g_{1i}g_{1l}g_{1k}g_{1j}
	= 
	0	\!\! \mod \langle f \rangle.
	\]
Since $f$ is irreducible and $g_{11} = \partial f /\partial x_1$ has smaller degree, $g_{11}$ is not a zero-divisor in $R[x_1, \hdots, x_n]/\langle f\rangle$. Therefore the minor $g_{ij}g_{kl} - g_{il}g_{kj}$ belongs to $\langle f \rangle$. 

From this it follows that $f^{k-1}$ divides the $k\times k$ minors of $G$ 
for every $2\leq k\leq n$, see \cite[Lemma 4.7]{PV13}. 
In particular, $f^{n-2}$ divides the entries of the adjugate matrix $G^{\rm adj}$.  
Let 
\begin{equation}\label{eq:Gadj}
M = (1/f^{n-2})\cdot G^{\rm adj}.
\end{equation}
Also $f^{n-1}$ divides $\det(G)$, and since these both have degree $n(n-1)$, 
there must be some constant $\lambda\in R$ for which $\det(G) = \lambda \cdot f^{n-1}$. 

We can see that $\lambda = 1$ by taking top degree terms. 
Since $\deg(f_i) = n-1$ and $\deg(g_{ij})\leq n-2$ for all $i\neq j$, the leading degree term of $\det(G)$ 	comes uniquely from the product of the diagonals $f_1\cdots f_n$ and is therefore $(\prod_{i=1}^n x_i)^{n-1}$. 
On the righthand side, the leading degree term of $f^{n-1}$ is also  
$(\prod_{i=1}^n x_i)^{n-1}$, showing that $\lambda=1$. 
Then 
	\[\det(M) \ \ = \ \ \frac{1}{f^{n(n-2)}}\cdot \det(G^{\rm adj})\ \ =\ \ 
	\frac{1}{f^{n(n-2)}} \det(G)^{n-1}\ \ =\ \   \frac{1}{f^{n(n-2)}} f^{(n-1)^2}\ \ =\ \  f.\]
	
Note that the entries of $M$ have degree $\leq (n-1)^2 - n(n-2) = 1$, so we can write 
	$M$ as $M_0+ \sum_{i=1}^n x_i M_i$ for some matrices $M_i\in R^{n\times n}$. 
	We claim that $\sum_{i=1}^n x_i M_i = \diag(x_1, \hdots, x_n)$. 
	
To see this, first note that a non-principal $(n-1)\times (n-1)$ minor of $G$ 
involves at most $n-2$ elements from the diagonal of $G$ 
and therefore has degree $\leq (n-2)(n-1) + (n-2) = n(n-2)$, since the 
off-diagonal entries of $G$ have degree $\leq n-2$. 
Therefore the off diagonal entries of $M$ have degree $\leq n(n-2) - n(n-2) =0$. 
	
Moreover in the expansion of any principal minor of $G$, there is a \emph{unique} term of degree $(n-1)^2$, namely the product of the leading terms of the diagonal elements, $\prod_{j\neq i} {\rm LT}(g_{jj})$. 
We can therefore take the leading terms \eqref{eq:Gadj} to find that
	\begin{align*}
		\sum_{i=1}^nx_iM_i  & = \frac{1}{\left({\rm LT}(f)\right)^{n-2}}\cdot \left({\rm diag}\left({\rm LT}(g_{11}), \hdots, {\rm LT}(g_{nn})\right)\right)^{\rm adj}\\
		& = \frac{1}{\left(\prod_{j=1}^nx_j\right)^{n-2}}\cdot \left(\prod_{j=1}^nx_j \cdot {\rm diag}\left(\frac{1}{x_1}, \hdots, \frac{1}{x_n}\right)\right)^{\rm adj}\\
%		& = \frac{\left(\prod_{j=1}^nx_j\right)^{n-1}}{\left(\prod_{j=1}^nx_j\right)^{n-2}}\cdot \left({\rm diag}\left(\frac{1}{x_1}, \hdots, \frac{1}{x_n}\right)\right)^{\rm adj}\\
%		& = \prod_{j=1}^n x_j \cdot {\rm diag}\left(\prod_{j\neq 1}\frac{1}{x_j}, \hdots,\prod_{j\neq n}\frac{1}{x_j}\right)\\
		& = {\rm diag}\left(x_1, \hdots, x_n\right).
	\end{align*}

Finally, for general $f$, we take a factorization of $f$ into irreducible polynomials $f = \prod_{\ell} f_{\ell}$. 
By Lemma~\ref{lem:factorDelta}, for every $i,j$, 
$\Delta_{ij}(f_{\ell})$ has a factorization $m_{ij}m_{ji}$ so into multiaffine polynomials $m_{ij}$ 
with ${\rm Res}_{x_k}(m_{ij},f) = m_{ik}m_{kj}$. By the arguments above, 
 $f_k$ has a determinantal representation of the correct form. 
Taking a block diagonal representation of these representations (and permuting the rows and columns if necessary to reorder $x_1, \hdots, x_n$) gives a determinantal representation for $f$. 
\end{proof}

\begin{remark}
Theorem~\ref{thm:DeltaDetRep} the matrix $G = (g_{ij})_{ij}$
and corresponding determinantal representation ${\rm diag}(x_1, \hdots, x_n)+A$ 
of $f$ satisfy 
\[
G  = ({\rm diag}(x_1, \hdots, x_n)+A)^{\rm adj} \text{ and } 
({\rm diag}(x_1, \hdots, x_n)+A) = f^{2-n}G^{\rm adj}.\]
\end{remark}

\begin{cor}\label{cor:Invariance}
Let $f = \det({\rm diag}(x_1,\hdots,x_n) +A)$ with $A\in {\rm Mat}_n(R)$
and $\gamma\in {\rm SL}_2(R)^n$. 
If $\beta = {\rm coeff}(\gamma\cdot f, \prod_{i=1}^nx_i)$ is nonzero, 
then for some $n\times n$ matrix $B$ with entries in $\tfrac{1}{\beta}R$, 
\[\gamma\cdot f = \beta \det({\rm diag}(x_1,\hdots,x_n) +B).\] 
\end{cor}

\begin{proof}
Let $g_{ij} \in R[{\bf x}]$ denote the $(i,j)$th entry of $({\rm diag}(x_1,\hdots,x_n) +A)^{\rm adj}$. 
We claim that $\frac{1}{\beta}\gamma\cdot f$ and  $\frac{1}{\beta}\gamma\cdot g_{ij}$ in $R(\tfrac{1}{\beta})[{\bf x}]$ satisfy the 
conditions in Theorem~\ref{thm:DeltaDetRep}. 
Here $\gamma$ acts of $f$ as a polynomial of multidegree ${\bf 1}_{[n]}$ 
and on $g_{ij}$ as a polynomial of multidegree ${\bf 1}_{[n]\backslash \{i,j\}}$. 

It is immediate that $\frac{1}{\beta}(\gamma \cdot f)\in R(\tfrac{1}{\beta})[{\bf x}]$ is multiaffine in $x_1, \hdots, x_n$ and has coefficient of $x_1\cdots x_n$ equal to one. 
We first note that 
\[
\Delta_{ij}(\tfrac{1}{\beta}(\gamma \cdot f)) 
= 
\tfrac{1}{\beta^2} (\gamma\cdot \Delta_{ij}(f)) = (\tfrac{1}{\beta} \gamma\cdot g_{ij})(\tfrac{1}{\beta} \gamma\cdot g_{ji})
\]
where $\gamma$ acts on $\Delta_{ij}(f)$ as a polynomial of multidegree $2\cdot {\bf 1}_{[n]\backslash \{i,j\}}$. 
By Proposition~\ref{prop:SL2_on_res},
\[
{\rm res}_{x_k}\left(\tfrac{1}{\beta}(\gamma \cdot g_{ij}), \tfrac{1}{\beta}(\gamma \cdot f) \right) 
= 
\tfrac{1}{\beta^2} (\gamma\cdot {\rm res}_{x_k}(g_{ij},f)) 
= \left(\tfrac{1}{\beta}\gamma \cdot g_{ik}\right)\left(\tfrac{1}{\beta}\gamma \cdot g_{kj}\right). 
\]
As the polynomials $\tfrac{1}{\beta}\gamma \cdot g_{ij}$ are multiaffine, this finishes the claim. 

By Theorem~\ref{thm:DeltaDetRep}, $\frac{1}{\beta}\gamma\cdot f$ 
equals $\det({\rm diag}(x_1,\hdots,x_n) +B)$ 
for some $B\in {\rm Mat}_n(R(\tfrac{1}{\beta}))$. 

We claim that $\beta B$ has entries in $R$.  By construction we have 
\[
{\rm diag}(x_1, \hdots, x_n)+B = (\tfrac{1}{\beta} \gamma \cdot f)^{2-n}(\tfrac{1}{\beta} \gamma \cdot G)^{\rm adj}
= 
\tfrac{1}{\beta}
(\gamma \cdot f)^{2-n}(\gamma \cdot G)^{\rm adj}.
\]
For the last equality, we use that $(\tfrac{1}{\beta})^{2-n}(\tfrac{1}{\beta})^{n-1} = \frac{1}{\beta}$. 
Multiplying by $\beta$ then gives 
\begin{equation}\label{eq:MatrixAction}
\beta \left({\rm diag}(x_1, \hdots, x_n)+B\right)  = 
(\gamma \cdot f)^{2-n}(\gamma \cdot G)^{\rm adj} \in {\rm Mat}_n(R[{\bf x}]),
\end{equation}
showing that the entries of $\beta B$ belong to $R$. 
\end{proof}

From this, we see that $\SL_2(\F)^n$ acts rationally on the set of matrices 
$A\in {\rm Mat}_n(\F)$  for $\F = {\rm frac}(R)$. 
Namely, if $f = {\rm diag}(x_1, \hdots, x_n)+A$ and 
$\gamma\in \SL_2(\F)^n$ with ${\rm coeff}(\gamma \cdot f, \prod_{i=1}^nx_i) = \beta \neq 0$, 
then as \eqref{eq:MatrixAction} in the proof of Corollary~\ref{cor:Invariance}, 
$\beta \left({\rm diag}(x_1, \hdots, x_n)+B\right)  = 
(\gamma \cdot f)^{2-n}(\gamma \cdot G)^{\rm adj}$ 
for some $B\in {\rm Mat}_n(\F)$. 
We can then define $\gamma\cdot A=B$.

Similarly, for a field $\F$, the mutiplicative group $(\F^*)^n$ acts on $n\times n$ matrices 
by diagonal conjugation. Namely, for $\lambda = (\lambda_1, \hdots, \lambda_n)$ 
we define 
\[
\lambda \cdot A := D^{-1}A D,
\]
where $D = {\rm diag} (\lambda_1, \hdots, \lambda_n)$.

\begin{prop}
The  action of ${\rm SL}_2(\F)^n$ on ${\rm Mat}_n(\F)$ commutes with diagonal conjugation.
\end{prop}
\begin{proof}
Let $A\in {\rm Mat}_n(\F)$ with $f = \det({\rm diag}(x_1, \hdots, x_n) + A)$ and $\gamma\in {\rm SL}_2(\F)^n$ for which 
${\rm coeff}(\gamma \cdot f, \prod_{i=1}^nx_i) = \beta\neq 0$. 
Let $G= (g_{ij})_{ij} = ({\rm diag}(x_1, \hdots, x_n) + A)^{\rm adj}$. 

For $\lambda \in (\F^*)^n$ and $D = {\rm diag}(\lambda_1, \hdots, \lambda_n)$,  
we see that $\frac{\lambda_j}{\lambda_i}(\gamma\cdot g_{ij})= \gamma\cdot(\frac{\lambda_j}{\lambda_i}g_{ij})$ and so 
$\gamma \cdot (D^{-1}GD) = D^{-1}(\gamma \cdot G)D$. 
Then 
\begin{align*}
{\rm diag}(x_1, \hdots, x_n)+D^{-1}(\gamma\cdot A) D
&= 
\alpha
(\gamma \cdot f)^{2-n}D^{-1}(\gamma \cdot G)^{\rm adj}D\\
&= 
\alpha
(\gamma \cdot f)^{2-n}(\gamma \cdot (D^{-1}GD))^{\rm adj}\\
&= 
{\rm diag}(x_1, \hdots, x_n)+\gamma\cdot (D^{-1}A D).
\end{align*}
\end{proof}

%%%%%%%%%%%%%%%%%%%%%%%%%%%%%%%%%%%%%%%%%%%%%%%%%%%%%%%%%%%%%%%%%%%%%%%%%%%%%%%%%%%%%
%%%%%%%%%%%%%%%%%%%%%%%%%%%%%%%%%%%%%%%%%%%%%%%%%%%%%%%%%%%%%%%%%%%%%%%%%%%%%%%%%%%%%

\section{Multiaffine algebra for constructing Hermitian factorizations }\label{sec:MA_algebra}
In this section, we develop an algorithm for constructing factorizations that satisfy the conditions 
in Theorem~\ref{thm:DeltaDetRep}. To do this, we find it most convenient to work in the following 
level of generality throughout this section. 
Let $S$ be a unique factorization domain with an automorphic involution $a\mapsto \overline{a}$. 
We use $0$ and $1$ to denote the additive and multiplicative identities of $S$. 
The map $S\to S$ given by $a\mapsto \overline{a}$ then must satisfy
\[
\overline{(\overline{a})} = a, \ \overline{0}=0, \ \overline{1}=1, \ \overline{a+ b} =\overline{a}+ \overline{b} 
\text{ and }  \overline{a\cdot b} =\overline{a}\cdot \overline{b}.
\]
for all $a,b\in S$. Let $R$ be the subring of elements fixed by this automorphism, that is $R = \{a\in S: \overline{a}=a\}$.

The example of interest is the ring $S = \C[x_{n+1}, \hdots, x_m]$ of polynomials with complex coefficients
with the involution given by complex conjugation. In this case the fixed ring is the subring of 
polynomials whose coefficients are real, $R = \R[x_{n+1}, \hdots, x_m]$.

\begin{assumptions}\label{assumptions}:
Let $f\in R[x_1, \hdots, x_n]$ satisfy the following: 
	\begin{enumerate}
		\item $f$ is irreducible in $R[x_1, \hdots, x_n]$,
		\item $f$ has degree $\leq 1$ in each variable $x_1, \hdots, x_n$, 
		\item the coefficient $\prod_{i=1}^n x_i$ in $f$ is nonzero, 
		\item for every $1\leq i<j\leq n$, $\Delta_{ij}(f)$ factors as $g_{ij} \overline{g_{ij}}$ in $S[x_1, \hdots, x_n]$, and 
		\item for every $1\leq i\leq n$, the partial derivative $\frac{\partial f}{\partial x_i}$ is irreducible in $R[x_1, \hdots, x_n]$ up to a constant. That is, for any factorization $\frac{\partial f}{\partial x_i} = g\cdot h$ in $R[x_1, \hdots, x_n]$,  $g\in R$ or $h\in R$.	\end{enumerate}
\end{assumptions}
In what follows, we will build up tools to show that under these assumptions 
Algorithm~\ref{alg:HerMatConst} produces the desired representation of $f$. 
We first exploit some properties of multiaffine polynomials.
For any disjoint subsets $S, T\subset [n]$, let 
\[f_S^T=\prod_{i\in S}\partial_i \cdot f |_{\{x_j = 0 \ : \ j\in T\}}.\] 
Note that if $f$ is multiaffine in $x_1, \hdots, x_n$, then for any $1\leq i < j \leq n$, we have 
\[
f  = x_i f_{i}  + f^{i}, \ \  f_i  = x_j f_{ij} +  f_i^j, \ \text{ and } \ f^i = x_j f^i_j +   f^{ij}.\]
From this, one can check that the formula for $\Delta_{ij}f$ can be written without  $x_i, x_j$: 
\[
\Delta_{ij}f = f_i^j\cdot  f_j^i -  f^{ij} \cdot f_{ij}
\]
If in addition we assume that $f$ and all its 
partial derivatives are irreducible, then $\Delta_{ij}(f)$ will have degree exactly $2$ in each variable,
as the following lemma shows.

\begin{lem}\label{lem:quad}
	If $f$ satisfies Assumptions~\ref{assumptions}, then for all $1\leq i,j \leq n$,  $\Delta_{ij}(f)$ is quadratic in each variable $x_k$ for $k \in [n]\setminus\{i,j\}$.
\end{lem}
\begin{proof}
	For $1\leq i<j\leq n$, we write $\Delta_{ij}(f)$ as a quadratic polynomial in the variable $x_k$:
	\[\Delta_{ij}(f) = f_i f_j - f_{ij} f = (f_{ik} x_k + f_i^k)(f_{jk} x_k + f_j^k) - (f_{ijk} x_k + f_{ij}^k)(f_k x_k +f^k),\]
	which gives 
	\[
	{\rm coeff}(\Delta_{ij}(f), x_k^2) = f_{ik} f_{jk} - f_{ijk} f_k  = \Delta_{ij}(f_k).
	\]
	If $\Delta_{ij}(f_k)=0$, then by Proposition~\ref{prop:DeltaZero}, $f_k$ is reducible, contradicting 
	Assumptions~\ref{assumptions}(5). 
\end{proof}

We next use ring maps given by taking resultants with $f$. 
For any $i=1, \hdots, n$, define 
\[\varphi_i: S[x_1, \hdots, x_m] \to S[x_k : k\neq i] \ \text{ by } \ 
\varphi_i(g) = {\rm Res}_{x_i}(g,f).
\]
For instance if we restrict to polynomials $g=g_jx_j+g^j$ with degree one in $x_j$, then
\[
\varphi_j(g, f) = - g_j f^j + g^j f_j.
\]
First we will start by listing some of the properties of these maps.
\begin{lem}\label{lem:Properties}
	If $f$ satisfies Assumptions~\ref{assumptions}, then, for all $g \in S[{\bf x}]$, the maps $\varphi_1, \hdots, \varphi_n$ satisfy the following:
	\begin{enumerate}
		\item \label{prop:evlf}$\varphi_j(f_i) = \Delta_{ij}(f)$ for all $1\leq i < j \leq n$,
		\item \label{prop:Deltas} $\varphi_j(\Delta_{ik}(f)) = \Delta_{ij}(f) \Delta_{jk}(f)$ for all distinct $1\leq i , j , k\leq n$,
		\item \label{prop:ind} if $\deg_j(g) = 0$, then $\varphi_j(g \cdot h) = g  \cdot \varphi_j(h)$ for all $h\in S[{\bf x}]$,	
		\item if $\deg_j(g)>0 $ and $\deg_j(h) >0$, then $\varphi_j(g \cdot h) = \varphi_j(g) \cdot \varphi_j(h)$ for all $  1\leq j \leq n$,
		\item \label{prop:comp} If $\deg_i(g) = \deg_j(g) = 1$ and $s g_j \notin \langle f_j\rangle$ for all $s \in S$, then $\varphi_j\circ \varphi_i(g) = \Delta_{ij}f \cdot \varphi_j(g)$,
		\item \label{prop:modf}  If $\deg_j(g) = 1$, $\varphi_j(g) \equiv f_j \cdot g$ modulo $\langle f\rangle$. 
	\end{enumerate}
	
\end{lem}
\begin{proof}
	We will prove (\ref{prop:comp}) and (\ref{prop:modf}) and all the other properties follow similarly by direct computations. 
	To prove property (\ref{prop:comp}), we write $g =g_{ij} x_i x_j + g_{i}^j x_i + g_{j}^i x_j + g^{ij}$, then
	\begin{align*}
		\varphi_j \circ \varphi_i(g)  &= \varphi_j(- g_{ij} x_j f^i- g_{i}^j f^i + g_{j}^i f_i x_j + g^{ij} f_i)\\ 
		&= \varphi_j\left((-g_{ij} f_j^i + g_j^i f_{ij})x_j^2 +(-g_{ij}f^{ij}+g_j^if_i^j + g^{ij} f_{ij} - g_i^j f_j^i)x_j +(g^{ij} f_i^j-g_i^j f^{ij})\right).
	\end{align*}
	Since for all $s \in S$, $sg_j \notin \langle f_j\rangle$, we see that ${\rm Coeff}_{x_j^2}(\varphi_i(g)) \neq 0$. Otherwise $g_j^i f_{ij} = g_{ij} f_j^i$, and since $f_j=f_{ij} x_i + f_j^i$ is irreducible up to a constant, then $f_{ij}$ and $f_j^i$ are relatively prime up to a constant $s \in S$. This implies that $f_j^i$ and $f_{ij}$ divide $sg_j^i$ and $sg_{ij}$ respectively and this implies that $sg_j \in \langle f_j\rangle$. 
	
	Applying the map $\varphi_j$ and simplifying then gives
	\[
	\varphi_j \circ \varphi_i(g)= \Delta_{ij}(f) (-g_j f^j + g^j f_j) = \Delta_{ij}(f) (\varphi_j (g))
	\]
	To prove (\ref{prop:modf}), we write $g$ as $g= g_jx_j + g^j$ and we use $f^j = f- f_j x_j$ 
	\[
	\varphi_j(g ) = - g_j f^j +  g^j f_j = -g_j (f - f_j x_j) + g^j f_j = -g_j f + f_j \ g
	\]
	Therefore  $\varphi_j(g) \equiv f_j\ g$ modulo $\langle f \rangle$.
\end{proof}
\begin{lem}\label{lem:twoconj}
	If $f$ satisfies Assumptions~\ref{assumptions} and $\Delta_{ij}(f) = p \overline{p}$ for some $1\leq i<j \leq n$, then for every $k \in [n]\setminus \{i,j\}$, 
	there is a factorization of each $\Delta_{ik}(f)$ and $\Delta_{jk}(f)$ into $q \overline{q}$ and $r\overline{r}$, respectively, such that $\varphi_k(p) = q r$.
\end{lem}

\begin{proof}
	Since $\Delta_{ik}(f)$ and $\Delta_{jk}(f)$ factor into two conjugates, we can write
	\[
	\Delta_{ik}(f) = a_1 \cdots a_s \cdot \overline{a_1}\cdots \overline{a_s} \ \ \ \ \text{ and } \ \ \ \ \Delta_{jk}(f) = b_1 \cdots b_t \cdot \overline{b_1} \cdots \overline{b_t}
	\]
	where $a_1, \hdots, a_s, b_1,\hdots, b_t$ are irreducible in $S[x_1, \hdots, x_m]$ that are multiaffine in $x_1, \hdots, x_n$. Then
	\[
	\varphi_k(p)\varphi_k(\overline{p}) =  \varphi_k(\Delta_{ij}(f)) = \Delta_{ik}(f)\Delta_{jk}(f) = a_1\cdots a_s\cdot \overline{a_1} \cdots \overline{a_s}\cdot b_1 \cdots b_t \cdot \overline{b_1} \cdots \overline{b_t}.
	\]
	After switching $a_i$ with $\overline{a_i}$ and $b_i$ with $\overline{b_i}$ if necessary, we get
	\[
	\varphi_k(p) = a_1\cdots a_s \cdot b_1 \cdots b_t = q \cdot r 
	\]
	where $q = a_1\cdots a_s$ and $r = b_1\cdots b_t$ are multiaffine polynomials such that $\Delta_{ik}(f) = q \overline{q}$ and $\Delta_{jk}(f) = r\overline{r}$ as desired.\end{proof}

\begin{lem}\label{lem:imcirc}
	If $f$ satisfies Assumptions~\ref{assumptions}
	and for some distinct $1\leq i,j,k \leq n$, the polynomials
	$\Delta_{ij}(f) = p \overline{p}$, $\Delta_{ik}(f) = q \overline{q}$ and $\Delta_{jk}(f) = r \overline{r}$ such that $\varphi_k(p) = q r$, then
	\[
	\varphi_j(q) = p \overline{r} \ \ \text{ and } \ \ \varphi_i(r) = p \overline{q}.
	\]
\end{lem}

\begin{proof}
	We will prove the first equality and the second holds similarly. First notice that since $\deg_i(p)=\deg_j(p)=0$,  $sp \notin \langle f_j \rangle$ for all $s \in S$.  Also, $\deg_{j}(r) = \deg_k(r)=0$. 
	Then using the properties in Lemma~\ref{lem:Properties} we get	
	\[ \varphi_j(q) r  = \varphi_j(q r) = \varphi_j\circ \varphi_k(p) = \Delta_{jk}(f) \cdot \varphi_j(p) = \Delta_{jk}(f) \cdot p.\]
	Since $\Delta_{jk}(f) = r \overline{r}$, dividing the above equation by $r$ gives the desired result. 
\end{proof}	

The following algorithm gives the desired factorizations of $\Delta_{ij}(f)$ into $g_{ij}\overline{g_{ij}}$ that
satisfy the hypothesis of Theorem \ref{thm:DeltaDetRep}, which will in turn give the desired 
Hermitian determinantal representation in Theorem~\ref{thm:HerDetRep}.

% remove [H] from below to un-force placement

\begin{algorithm}[H]\caption{Compatible Hermitian factorizations of Rayleigh differences \smallskip \\
		\textbf{Input:} $f\in R[x_1, \hdots, x_n]$ satisfying Assumptions~\ref{assumptions}  \\
		\textbf{Output:} Polynomials $\{g_{jk} : 1\leq j < k \leq n\}$ in $S[x_1, \hdots, x_n]$\smallskip
	}\label{alg:HerMatConst}
	\begin{algorithmic}[0]
		\State Take $g_{12}\in \K[x_1, \hdots, x_n]$ so that $\Delta_{12}f = g_{12}\cdot \overline{g_{12}}$
%		\For{$k = 1, k\leq n, k${\footnotesize $++$}}
%		\State $g_{kk} := f_k$
%		\EndFor
		\For {$k = 3,k\leq n, k${\footnotesize $++$}}
		\State  $Q_0 := \gcd\{\Delta_{1k}(f) , \varphi_k(g_{12}), \hdots, \varphi_k(g_{1(k-1)})\}$\smallskip 
		\State  Factor $\Delta_{1k}(f) = p_{k,1} \cdots p_{k,m_k} \cdot \overline{p_{k,1}} \cdots \overline{p_{k,m_k}}$ with $p_{k,j}$ irreducible for all  $j$
		\For{$j = 1, j \leq m_k, j${\footnotesize $++$}}
		\If{$p_{k,j} \overline{p_{k,j}}$ divides $Q_{j-1}$} $Q_{j} := Q_{j-1}/\overline{p_{k,j}}$
		\Else  \ $Q_{j} := Q_{j-1}$
		\EndIf
		\EndFor
		\State $g_{1k}:= Q_{m_k}$
		\For{$j = 2, j \leq k-1, j${\footnotesize $++$}} \smallskip 
		\State $g_{jk} := \overline{\varphi_k(g_{1j})/g_{1k}}$
		\EndFor
		\EndFor
	\end{algorithmic}
\end{algorithm}

\begin{prop}\label{prop:algJust}
	The polynomials $\{g_{ik}\}_{1 \leq i <k \leq n}$ constructed in Algorithm~\ref{alg:HerMatConst} satisfy
	\begin{itemize}
		\item[(a)] $g_{1k}$ is multiaffine in  $x_1,\hdots,x_n$ for all $k>1$,
		\item[(b)] $\varphi_k(g_{1i})= g_{1k} \overline{g_{ik}}$ for all $1<i<k$, and 
		\item[(c)] $\Delta_{ik}(f) = g_{ik} \overline{g_{ik}}$ for all $1\leq i< k$.
	\end{itemize}
\end{prop}
\begin{proof}
	(a) This is immediate for $k = 2$. For $2 < k\leq n$, notice that $\Delta_{1k}(f)$ has 
	degree two in $x_1, \hdots, x_n$. Let $\ell \in [n]\backslash\{1,k\}$ and 
	let $p_{k,j}, \overline{p_{k,j}}$ be the unique 
	irreducible factors of $\Delta_{1k}(f)$ with degree one in $x_{\ell}$. 
	By construction, $g_{1k}$ divides $Q_{j}$, which in turn divides $\Delta_{1k}(f)/ \overline{p_{k,j}}$. 
	Since this quotient only has degree one in $x_{\ell}$, $g_{1k}$  must have degree $\leq 1$ in $x_{\ell}$. 
	
	(b) follows directly from construction.  
	
	(c) We proceed by induction on $k$. It is trivially true for $k = 2$. For the inductive step, we will prove the claim for $\Delta_{1k}(f)$ and the other cases follow. By construction, $g_{1k} \overline{g_{1k}}$ divides $\Delta_{1k}(f)$. 
	To see this, note that for each $j=1, \hdots, m_k$ in Algorithm~\ref{alg:HerMatConst}, we can take $q_{j} = p_{k,j}$ if
	$p_{k,j}$ divides $g_{1k}$ and $q_{j} = \overline{p_{k,j}}$ otherwise.  Then, by construction, $g_{1k}$ divides $q = \prod_{j=1}^{m_k}q_j$ and $q\cdot\overline{q} = \Delta_{1k}(f)$, showing that $g_{1k}\cdot \overline{g_{1k}}$ divides 
	$\Delta_{1k}(f)$. 
	
	Suppose for the sake of contradiction that $\Delta_{1k}(f) \neq g_{1k} \overline{g_{1k}}$. Then there is some irreducible factor $p$ of $\Delta_{1k}(f)$ such that $p \overline{p}$ does not divide $g_{1k}\overline{g_{1k}}$. 
	We claim that for every $1< i<k$, either $p$ or $\overline{p}$ divides $\varphi_k(g_{1i})$. 
	By induction, for $1<i<k$, $g_{1i} \overline{g_{1i}}=\Delta_{1i}(f)$. Applying $\varphi_k$ gives 
	\[
	\varphi_k(g_{1i})\cdot \varphi_k(\overline{g_{1i}}) = \varphi_k(\Delta_{1i}(f)) = \Delta_{1k}(f)\cdot  \Delta_{ik}(f). 
	\]
	Since $p$ is irreducible and divides $\Delta_{1k}(f)$, it must divide either $\varphi_k(g_{1i})$ or 
	$\varphi_k(\overline{g_{1i}})=\overline{\varphi_k(g_{1i})}$. In the latter case, $\overline{p}$ divides $\varphi_k(g_{1i})$.
	Since neither $p$ nor its conjugate divide $g_{1k}$, it follows from the construction that 
	neither $p$ nor $\overline{p}$ divide $Q_0 = \gcd\{\Delta_{1k}(f) , \varphi_k(g_{12}), \hdots, \varphi_k(g_{1(k-1)})\}$.  Hence there exists distinct
	$2 \leq i, j <k$ such that neither $p$ divide $\varphi_k(g_{1i})$ nor $\overline{p}$ divide $\varphi_k(g_{1j})$. By switching $p$ and $\overline{p}$ if needed, we can assume $i< j$. 
	
	By induction (on $k$), we know that $\Delta_{1i}(f) = g_{1i} \overline{g_{1i}}$, $\Delta_{1j}(f) = g_{1j} \overline{g_{1j}}$ and $\Delta_{ij}(f) = g_{ij} \overline{g_{ij}}$. 
	Moreover, by (b), 	$\varphi_j(g_{1i})= g_{1j} \overline{g_{ij}}$.
	Lemma~\ref{lem:imcirc} then  implies that $\varphi_1(g_{ij}) = g_{1i} \overline{g_{1j}}$ and
	\[
	\Delta_{1k}(f) \varphi_k(g_{ij}) = \varphi_k(\varphi_1(g_{ij})) = \varphi_k(g_{1i} \overline{g_{1j}})= \varphi_k(g_{1i}) \varphi_k(\overline{g_{1j}}).
	\]
	Now neither $\varphi_k(g_{1i})$ nor $\varphi_k(\overline{g_{1j}})=\overline{\varphi_k(g_{1j})}$ is divisible by $p$ while $p$ divides $\Delta_{1k}(f)$ and this gives the desired contradiction. Therefore $\Delta_{1k}(f) = g_{1k} \overline{g_{1k}}$. 
	
	For $1<i<k$, we calculate that
	\[g_{ik}\cdot \overline{g_{ik}} = \frac{\varphi_k(\overline{g_{1i}})}{\overline{g_{1k}}} \cdot 
	\frac{\varphi_k(g_{1i})}{g_{1k}}
	= \frac{\varphi_k(\overline{g_{1i}}g_{1i})}{\Delta_{1k}(f)} = \frac{\Delta_{1k}(f)\Delta_{ik}(f)}{\Delta_{1k}(f)} = \Delta_{ik}(f). \]	
\end{proof}

\begin{cor}\label{cor:IrredDetRep}
	If $f\in R[x_1, \hdots, x_n]$ satisfies Assumptions~\ref{assumptions}, then there exists a factorization of $\Delta_{ij}(f)$ into $g_{ij}g_{ji}$ such that $g_{ij} \in S[x_1,\hdots,x_n]$, $g_{ji} = \overline{g_{ij}}$, and $\varphi_k(g_{ij}) = g_{ik} g_{kj}$  for all distinct $1\leq i,j,k \leq n$.
\end{cor}	
\begin{proof} Let $\{g_{ij}: 1\leq i< j\leq n\}$ be the polynomials given by Algorithm~\ref{alg:HerMatConst}  and for $i<j$ let $g_{ji} = \overline{g_{ij}}$. 
	By Proposition~\ref{prop:algJust}, $\Delta_{ij}(f)= g_{ij}\overline{g_{ij}} = g_{ji}g_{ij}$. 
	Since $\Delta_{ij}(f)$ is quadratic in each variable $x_1, \hdots, x_n$, then 
	$g_{ij}$ is multiaffine in $x_1, \hdots, x_n$. We will show that $\varphi_k(g_{ij}) = g_{ik}g_{kj}$ for all distinct $i,j,k$. Assuming that $i<j<k$ and using Proposition~\ref{prop:algJust} we get
	\begin{align*}	
		{\rm res}_{x_k}(g_{1i},f) &= \varphi_k(g_{1i}) = g_{1k}\overline{g_{ki}} = g_{1k}g_{ki}, \text{ and }\\
		{\rm res}_{x_k}(g_{j1},f) &= \varphi_k(g_{j1}) = \overline{\varphi_k(g_{1j})} =\overline{ g_{1k}\overline{g_{kj}}} = g_{jk}g_{k1}.
	\end{align*}
	Multiplying the above two equations and using Properties~\ref{lem:Properties} we get
	\[
	\varphi_k(g_{1i}g_{j1}) = \Delta_{1k}(f) g_{ki} g_{jk}.
	\]
	Using Proposition~\ref{prop:algJust} again, we know that $\varphi_j(g_{1i}) = g_{1j} g_{ji}$ and Lemma~\ref{lem:imcirc} implies that $\varphi_1(g_{ji}) = g_{j1} g_{1i}$. Again using Properties~\ref{lem:Properties} we find that 
	\[
	\Delta_{1k}(f) \varphi_k(g_{ji}) = \Delta_{1k}(f) g_{ki} g_{jk}.
	\]
	Since $f$ is irreducible, $\Delta_{1k}(f)$ is nonzero and we conclude that $\varphi_k(g_{ij}) = \overline{\varphi_k(g_{ji})}=\overline{g_{ki} g_{jk}} = g_{ik} g_{kj}$. Using Lemma~\ref{lem:imcirc}, we get that $\varphi_j(g_{ik}) = g_{ij} g_{jk}$ and $\varphi_i(g_{jk}) = g_{ji} g_{ik}$ as desired. 
	\end{proof}

\begin{example}\label{ex:HerEx}($n=4)$. Consider  $f \in \mathbb{R}[x_1,x_2,x_3,x_4]$ given by
	\[
	f(x_1,x_2,x_3,x_4) = x_1 x_2 x_3 x_4 - x_1 x_2 -  x_1 x_3 -  x_1 x_4 - x_2 x_3 - x_2 x_4 -  x_3 x_4 +1
	\]
	For any distinct $i,j,k, \ell \in [4]$, the Raleigh differences of $f$ with respect to $x_i$ and $x_j$ is
	\[
	\Delta_{ij}(f) = (x_k^2 +1)(x_{\ell}^2+1) = (x_k - \ci)(x_k + \ci) (x_{\ell} + \ci)(x_{\ell} -\ci).
	\]
	Using Algorithm~\ref{alg:HerMatConst}, we can choose $g_{12}$ as any multiaffine factor of $\Delta_{12}(f)$ of degree two. There are two possibilities, namely $g_{12} = (x_3 - \ci) (x_4-\ci)$ or $g_{12} = (x_3 - \ci)(x_4 +\ci)$ and one can check that either choice works. We will start with the first option and compute
	\[
	\varphi_3(g_{12}) = -\ci (x_1+\ci) (x_2+\ci) (x_4-\ci) (x_4 +\ci).
	\]
	To choose $g_{13}$ we compute $\gcd(\Delta_{13}(f),\varphi_3(g_{12})) = -\ci(x_2+\ci)(x_4-\ci)(x_4+\ci)$. Thus, up to a constant, we have two choices for $g_{13}$, namely $-\ci(x_2 + \ci) ( x_4 - \ci)$ or $-\ci(x_2 + \ci)(x_4 + \ci)$.
	We will choose the first option, giving 
	\[
	g_{23} = \overline{(\varphi_3(g_{12})/g_{13})} = (x_1+\ci)(x_4+\ci).
	\] 
	To find $g_{14}$, we compute the $\gcd(\Delta_{14}(f), \varphi_4(g_{12}),\varphi_4(g_{13})) = -\ci(x_2+\ci)(x_3 + \ci)$ and we get $g_{24}$ and $g_{34}$ similarly. The final matrix is $M =$
	\[
	 \small\begin{pmatrix}
		-x_2 x_3 x_4-x_2-x_3-x_4 & (x_3 - \ci) (x_4-\ci) & -\ci(x_2 + \ci) ( x_4 - \ci) &  -\ci(x_2+\ci)(x_3 + \ci) \\
		(x_3 + \ci) (x_4+\ci)	&  x_1 x_3 x_4-x_1-x_2-x_3 & (x_1+\ci)(x_4-\ci) & (x_1-\ci)(x_3+\ci)\\
		\ci(x_2-\ci)(x_4+\ci) & (x_1-\ci)(x_4+\ci) &  x_1 x_2 x_4 -x_1 - x_2 - x_4& \ci(x_1-\ci)(x_2-\ci)\\
		\ci(x_2-\ci)(x_3 - \ci) & (x_1+\ci)(x_3-\ci) & -\ci(x_1+\ci)(x_2+\ci) & x_1 x_2 x_3 -x_1 - x_2 - x_3
	\end{pmatrix}.
	\]
	Now we compute the adjugate matrix of $M$ and divide its entries by $f^2$ we get 
	\[
	A= \frac{1}{f^2}M^{\rm adj} = \small\begin{pmatrix}
		x_1 & -1 & \ci & \ci \\ -1 & x_2 & -1 & -1 \\ -\ci & -1 & x_3 & -\ci \\ -\ci & -1 & \ci & x_4
	\end{pmatrix}
	\]
	and one can check that $\det(A) = f$. The algorithm gives all the possible representations of $f$, up to diagonal equivalence, namely
	\[
	\left\{
	\small\begin{pmatrix}
		x_1 & -1 & \ci & \ci \\ -1 & x_2 & -1 & -1 \\ -\ci & -1 & x_3 & -\ci \\ -\ci & -1 & \ci & x_4
	\end{pmatrix}, \ \
	\small\begin{pmatrix}
		x_1 & -1 & \ci & \ci \\ -1 & x_2 & -1 & -1 \\ -\ci & -1 & x_3 & \ci \\ -\ci & -1 & -\ci & x_4
	\end{pmatrix}, \ \
	\small\begin{pmatrix}
		x_1 & -1 & \ci & -\ci \\ -1 & x_2 & -1 & -1 \\ -\ci & -1 & x_3 & -\ci \\ \ci & -1 & \ci & x_4
	\end{pmatrix}
	\right\}.
	\]
\end{example}

\section{Hermitian determinantal representations}\label{sec:otherDetRep}
Let $\K$ be a field with an automorphism $a\mapsto \overline{a}$ of order two. 
Let $\F$ be the fixed field of this automorphism. 
We call a matrix $A\in {\rm Mat}_n(\mathbb{K})$ Hermitian if $A = \overline{A}^T$. 

\subsection{Consequences of Algorithm~\ref{alg:HerMatConst} }

\begin{thm}\label{thm:HerDetRep}
	Let $f \in \F[x_1, \hdots, x_m]$ be a polynomial of total degree $n\leq m$ that is multiaffine in 
	$x_1, \hdots, x_n$ and coefficient of $x_1\cdots x_n$ equals to one. 
		There exist Hermitian matrices $A_{n+1}, \hdots, A_m, A_0$ so that 
	\[
	f = \det\left( {\rm diag}(x_1, \hdots, x_n) +\sum_{j=n+1}^mx_jA_j +A_0\right)
	\]
	if and only if for all $i,j\in [n]$, $\Delta_{ij}(f)$ is a Hermitian square in $ \K[x_1, \hdots, x_m]$.
\end{thm}

\begin{lem}\label{lem:Action2Irred}Let $\F$ be an infinite field and 
$f\in \F[x_1, \hdots, x_m]$ be irreducible, multiaffine in the variables $x_1,\hdots,x_n$ and have coefficient of $x_1\cdots x_n$ equals to $1$. Let $R = \F[x_{n+1},\hdots,x_m]$. For a generic element $\gamma \in \SL_2(\F)^n$, 
the derivatives $\frac{\partial}{\partial x_j}(\gamma \cdot f)$ are irreducible in $R[x_1,\hdots,x_n]$ for $j=1, \hdots, n$, up to a constant and the coefficient of $\prod_{i=1}^n x_i$ is nonzero.
\end{lem}
\begin{proof} 
Consider $\gamma_j = \footnotesize{\begin{pmatrix} a && b \\ c && d	\end{pmatrix}} \in \SL_2(\F)$ acting on $x_j$. 
Then $\partial_j(\gamma \cdot f)= a f_j + c f^j$ where $f = f_j x_j + f^j$.
Consider the set 
\[	\mathcal{X} = \big\{(a,c)\in \F^2: a f_j + c f^j \text{ is reducible in $R[x_1,\hdots,x_n]$ up to constants} \},
\]
Suppose that the multidegree of $f$ in $x_1, \hdots, x_m$ is given by ${\bf d}\in \N^m$. By assumption, 
${\bf d}_i =  1$ for $i=1, \hdots, n$. 
Note that $\mathcal{X}$ is contained in the union $ \bigcup_{{\bf e}}\mathcal{X}_{{\bf e}}$ where 
	\[
	\mathcal{X}_{{\bf e}} =\big\{(a,c) \in \F^2: af_j + c f^j \in \F^{\rm alg}[x_1, \hdots, x_m]_{\leq {\bf e}} \cdot \F^{\rm alg}[x_1, \hdots, x_m]_{\leq {\bf d}-{\bf e}}  \}
	\]	
and the union is taken over all vectors ${\bf e}\in \N^m$ that are coordinate-wise $\leq {\bf d}$ 
with the property that ${\bf e}_i = 1$ and ${\bf e}_k = 0$  for some $i,k\in[n]\backslash\{j\}$.  
Here $\F^{\rm alg}$ denotes the algebraic closure of $\F$. 
To see this, suppose $(a,c)\in \mathcal{X}$, meaning $af_j + c f^j = g\cdot h$ where $g,h\in R[x_1, \hdots, x_n]$ 
are not constants (i.e. elements of $R$). In particular, for some $i,k\in [n]\backslash\{j\}$, 
$\deg_{i}(g)>0$ and $\deg_k(h)>0$. Since $af_j + c f^j$ has degree at most one in each of $x_i$ and $x_k$, 
it follows that $\deg_{i}(g)=\deg_k(h)=1$ and $\deg_{k}(g)=\deg_i(h)=0$. 
Taking ${\bf e}\in \N^m$ to be the multidegree of $g$ gives $g\in \F[x_1, \hdots, x_m]_{\leq {\bf e}}$ 
and $h\in \F[x_1, \hdots, x_m]_{\leq {\bf d}-{\bf e}}$. Then $(a,c)\in  \mathcal{X}_{{\bf e}}$.

By the projective elimination theorem, the image $\F^{\rm alg}[x_1, \hdots, x_m]_{\leq {\bf e}} \cdot \F^{\rm alg}[x_1, \hdots, x_m]_{\leq {\bf d}-{\bf e}}$ is Zariski-closed in the vectorspace $\F^{\rm alg}[x_1, \hdots, x_m]_{\leq {\bf d}}$. Intersecting with the 
$\F$-subspace spanned by $\{f_j,f^j\}$ shows that  $\mathcal{X}_{{\bf e}}$ and hence $\cup_{\bf e}\mathcal{X}_{{\bf e}}$ is Zariski-closed in $(\F)^2$. Therefore this union is either all of $\F^2$ or is an algebraic set of codimension $\geq 1$. 
Suppose, for the sake of contradiction, that it is all of $\F^2$. 
Then there exists some ${\bf e}$ for which $\mathcal{X}_{{\bf e}}= \F^2$. 
By assumption, there are $i,k\in [n]$ so that ${\bf e}_i =1$ and $({\bf d} - {\bf e})_k=1$. 
By  Proposition~\ref{prop:DeltaZero}, it follows that for all $a,c\in \F$, 
$\Delta_{ik}(af_j + cf^j)$ is identically zero.  For $c=1$, this corresponds to the evaluation of $\Delta_{ik}(f)$ at $x_j=a$. 
It follows that the polynomial $\Delta_{ik}(f)$ is identically zero (see e.g. \cite{AV21}).  Proposition~\ref{prop:DeltaZero} then implies that $f$ factors as the product of two nonconstant elements in $R[x_1, \hdots, x_n]$. 
\end{proof}

\begin{proof}[Proof of Theorem~\ref{thm:HerDetRep}]
	First assume that $\F$ is an infinite field and that $f$ is irreducible in $\F[x_1, \hdots, x_m]$. Let $R=\F[x_{n+1},\hdots,x_m]$.
	By Lemma~\ref{lem:Action2Irred}, 
	there exists a generic $\gamma\in  \SL_2(\F)^n$ such that the partial derivatives 
	of $\gamma\cdot f$ are all irreducible in $R[x_1, \hdots, x_n]$ and the coefficient of $x_1 \cdots x_n$ in $\gamma\cdot f$ is nonzero. Then by Corollary~\ref{cor:IrredDetRep}, there exists $\{g_{ij}\}_{ 1\leq i\neq  j\leq n}$ with $g_{ij}$ in $\K[x_\ell:\ell\neq i,j]$ satisfying $\Delta_{ij}(\gamma\cdot f) = g_{ij} g_{ji}$, $g_{ji} = \overline{g_{ij}}$,  and ${\rm res}_{x_k}(g_{ij},\gamma\cdot f) = g_{ik} g_{kj}$ for all distinct $ i,j,k\in [n]$. Acting by $\gamma^{-1}$ on $\gamma \cdot f$ and using Proposition~\ref{prop:SL2_on_res} we get $h_{ij} = \gamma^{-1}\cdot g_{ij}$ such that $\Delta_{ij}(f) = h_{ij} \overline{h_{ij}}$ and ${\rm res}_{x_k}(h_{ij},f) = h_{ik} h_{kj}$ and thus using Theorem~\ref{thm:DeltaDetRep} we get a determinantal representation of $f$ over $\K$ and since $h_{ji} = \overline{h_{ij}}$, the matrix will be Hermitian.
	
	Now suppose that $f$ is reducible. Let $g$ be an irreducible factor of $f$ with $\deg_i(g) =1$ for $i\in I\subseteq [n]$ 
	and $\deg_j(g) = 0$ for $j\in [n]\backslash I$  
	where the coefficient of $\prod_{i\in I}x_i$ equals one. 
	By Lemma~\ref{lem:factorDelta}, for every $i,j\in I$,  $\Delta_{ij}(g)$ is a Hermitian square. 
	Therefore $g$ has a determinantal representation of the correct form,
	$g = \det({\rm diag}(x_i : i\in I) + \sum_{j=n+1}^mx_jA_{ij} + A_{i0})$. 
	Taking a block diagonal representation of these representations (and permuting the rows and columns 
	to reorder $x_1, \hdots, x_n$) gives a determinantal representation for $f$. 
	
	Now suppose that $\F$ is a finite field. Consider the transcendental extension of $\F$ to $\F(t)$ and of $\K$ to $\K(t)$. Then by the arguments above,  $f  = \det\left(\diag{x_1,\hdots,x_n}+A(t)\right)$ for some Hermitian matrix $A(t) \in {\rm Mat}_n(\K(t))$. 
	The $(i,j)$th entry of $A(t)$ can be written as  $a_{ij} = \frac{p_{ij}}{q_{ij}}$ where $p_{ij},q_{ij} \in \K[t]$ are relatively prime and 
	the polynomial $q_{ij}$ is nonzero. 
	Specializing to $t=0$ will give a determinantal representation of $f$ over $\K$. 
	To do this, we need to check that $q_{ij}(0)$ is nonzero for all $i,j$.
	If $a_{ij}=0$, then we can take $p_{ij}=0$ and $q_{ij}=1$. Suppose that for some $i,j\in [n]$, $p_{ij}$ is nonzero and  
	$q_{ij}(0) = 0$. Then $t$ divides $q_{ij}$ and so also divides $\overline{q_{ij}}$. Notice that
	\[
	a_{ii} = {\rm coeff}(f,\prod_{k\neq i} x_k) \ \  \text{ and } \ \  a_{ij} a_{ji} =  a_{ii} a_{jj} - {\rm coeff}\left(f,\prod_{k\neq i,j} x_k\right) 
	\]
	are both in $\F$ and hence $p_{ij} \overline{p_{ij}} = r q_{ij} \overline{q_{ij}}$ for some $r \in \F^*$. We get the desired contradiction by noticing that $t^2$ divides the left-hand side of the equation, while it does not divide the right-hand side since $p_{ij}$ and $q_{ij}$ are relatively prime. Therefore we can specialize both sides of the equation $f  = \det\left(\diag{x_1,\hdots,x_n}+A(t)\right)$ to $t=0$, which gives a Hermitian determinantal representation of $f$. 
\end{proof}

\begin{example} Consider the polynomial 
$f = x_1x_2x_3 + x_1+x_2+x_3 + 1$ over the field $\F = \F_2$.
The Rayleigh difference $\Delta_{12}(f) = x_3^2+x_3+1$ does not factor in $\F_2[x_3]$, 
showing that the coefficient vector of $f$ is not in the image of $\Mat_3(\F_2)$ under the principal minor map. 

Consider the quadratic extension $\K = \F_2[\alpha]/\langle \alpha^2+\alpha+1\rangle$. 
The map $\alpha \mapsto 1+\alpha$ extends to an automorphic involution on $\K$ that fixes $\F_2$. 
Over $\K$, the Rayleigh differences factor into multiaffine polynomials, namely 
$\Delta_{ij}(f) = (x_k + \alpha)(x_k+1+\alpha)$, for distinct $i$, $j$, $k$. As then guaranteed by 
Theorem~\ref{thm:HerDetRep}, $f$ has a Hermitian determinantal representation over $\K$: 
\[
f = \det\begin{pmatrix} 
x_1 & 1+\alpha & 1+\alpha \\ 
\alpha & x_2 & 1+\alpha \\ 
\alpha & \alpha & x_3
\end{pmatrix}.
\]
\end{example}

\begin{cor}\label{cor:HerCompDetRep}
	Let $f \in \R[x_1, \hdots, x_m]$ be a polynomial of total degree $n\leq m$ that is multiaffine in 
	$x_1, \hdots, x_n$ and coefficient of $x_1\cdots x_n$ equals to one. 
	There exist Hermitian matrices $A_{n+1}, \hdots, A_m, A_0$ so that 
	\[
	f = \det\left( {\rm diag}(x_1, \hdots, x_n) +\sum_{j=n+1}^mx_jA_j +A_0\right)
	\]
	if and only if for all $i,j\in [n]$, $\Delta_{ij}(f)$ factors as $g_{ij} \overline{g_{ij}}$ for $g_{ij}\in \C[x_1, \hdots, x_m]$.
\end{cor}

This provides a partial converse to \cite[Corollary 4.3]{KPV15}, which states that 
if some power of a polynomial $f$ has a definite determinantal representation, then 
for all $i,j$, the Rayleigh difference $\Delta_{ij}(f)$ is a sum of squares.  
In particular, Hermitian representations of $f$ give real symmetric determinantal representations of $f^2$. 
We might hope for the following. 

\begin{conjecture}
If $f \in \R[x_1, \hdots, x_m]$ is multiaffine in $x_1, \hdots, x_n$ and coefficient of $x_1\cdots x_n$ is nonzero, 
then some power of $f$ has a definite real symmetric determinantal representation if and only if for all $i,j$, $\Delta_{ij}(f)$ 
is a sum of squares in $\R[x_1, \hdots, x_m]$. 
\end{conjecture}

\subsection{Other multiaffine determinantal representations}\label{sec:otherMADetRep}
In this section we restrict ourselves to fields and consider the set of multiaffine determinantal polynomials of the form
%A \emph{definite determinantal representation} of $f \in \F[{\bf x}]_{\rm MA}$ is an expression  
\begin{equation}\label{eq:det}
	f({\bf x}) = \lambda \det\left( V {\rm diag}(x_1, \hdots, x_m) V^* + W\right) =\lambda \det\left(\sum_{i=1}^m x_i v_iv_i^* + W\right)
\end{equation}
for some $\lambda\in \F$, some matrix $V = (v_1, \hdots, v_m)\in \K^{n\times m}$ and some $n\times n$ Hermitian matrix $W$. 
Note that when we take $V$ to be the $n\times n$ identity matrix and $\lambda = 1$, 
this is the principal minor polynomial $f_W$. 
When $n<m$, the coefficient of $x_1\cdots x_m$ in $f$ is necessarily zero.

\begin{thm}\label{thm:GenDet}
	A polynomial $f\in \F[{\bf x}]_{\rm MA}$ has a determinantal representation \eqref{eq:det} if and only if 
	for all $i,j\in [n]$, $\Delta_{ij}f$ is Hermitian square in $\K[{\bf x}]$. 
	Moreover, one can always take a representation of size $n = \deg(f)$ in  \eqref{eq:det}. 
\end{thm}

\begin{proof}
	($\Rightarrow$) 
	Without loss of generality, we show that $\Delta_{12}(f)$ is a Hermitian square. 
	First suppose $v_1$ and $v_2$ are linearly dependent, i.e.~let $v_1 = \alpha v_2$ for some $\alpha\in \K$. 
	Then $v_1v_1^* = \alpha \overline{\alpha} v_2v_2^*$ and $f(x_1, \hdots, x_m) = f(0, \alpha \overline{\alpha}  x_1 +x_2, x_3, \hdots, x_m)$. Taking partial derivatives shows that $\frac{\partial f}{\partial x_1} =  \alpha \overline{\alpha} \frac{\partial f}{\partial x_2}$ and that $\frac{\partial^2 f}{\partial x_1\partial x_2} = 0$.  Then $\Delta_{12}(f) = (\alpha\frac{\partial f}{\partial x_2})( \overline{\alpha}\frac{\partial f}{\partial x_2})$ is a Hermitian square. 
	
	If $v_1$ and $v_2$ are linearly independent, 
	then there is an invertible matrix $U \in \K^{n\times n}$ with $Uv_1 = e_1$ and 
	$Uv_2 = e_2$. 
	Then 
	\[
	|\det(U)|^2f = 
	\lambda\det\left(
	U\left(\sum_{i=1}^m x_iv_iv_i^* + W\right)U^{*}\right)
	= \lambda\det\left(
	{\rm diag}(x_1,x_2, {\bf 0}) + 
	\sum_{i=3}^m x_i\widetilde{v_i}\widetilde{v_i}^*  + \widetilde{W}\right).
	\]
	where $\widetilde{v}_i = Uv_i$ and $\widetilde{W} = UWU^{*}$.
	These matrices are still Hermitian and so by equation \eqref{eq:dodgson}, $\Delta_{12}(f)$ is Hermitian square.  
	
	($\Leftarrow$)
	Let $d = \deg(f)$. 
	We can assume, without loss of generality, that the coefficient  of $x_1\cdots x_d$ in $f$ is nonzero. 
	Moreover since the set of polynomials of the form  \eqref{eq:det} is invariant under scaling, we can assume
	that this coefficient equals one. 
	By Theorem~\ref{thm:HerDetRep}, there are Hermitian matrices $A_0, A_{d+1}, \hdots, A_{m}$ 
	so that
	\[f = \det\left({\rm diag}(x_1, \hdots, x_d) + \sum_{j=d+1}^m x_jA_j + A_0\right).\]  
	We take $W = A_{0}$. 
	By Lemma~\ref{lem:rank} below, for every $k=n+1, \hdots, m$, the rank of $A_k$ equals the degree of $f$ in $x_k$, which is one. 
	It remains to show that the matrix $A_k$ has the form $v_kv_k^*$ for some $v_k\in \K^d$. 
 
	By homogenizing and specializing variables to zero, it suffices to consider polynomials 
	of the form $f = \det\left({\rm diag}(x_1, \hdots, x_d) + x_{d+1}A\right)$ where $A\in \K^{n\times n}$ is Hermitian and rank-one. 
	Then $f = \prod_{i=1}^dx_i + \sum_{j=1}^dA_{jj} \prod_{i\in [d]\backslash\{j\}}x_i$, where $A_{jj}$ is the $j$th entry of $A$. 
	Then for $j=1, \hdots, d$, 
	\[
	\Delta_{j(d+1)}(f) = f_j^{d+1}f_{d+1}^j - f_{j(d+1)}f^{j(d+1)}  
	= \left(\prod_{i=1}^d x_i\right)\left(A_{jj}\prod_{i\in [d]\backslash\{j\}} x_i\right)
	= A_{jj}\left(\prod_{i\in [d]\backslash\{j\}} x_i\right)^2.
	\]
	By assumption, $\Delta_{j(d+1)}(f)$ is a Hermitian square, and so we see that $A_{jj} = \alpha_{j}\overline{\alpha_{j}}$ for 
	some $\alpha_{j}\in \K$. Since $A$ has rank one, we can write it as $\lambda uu^*$ for some $\lambda\in \F^*$ and $u\in \K^n$.
	If $u_j\neq 0$, then $\lambda u_j\overline{u_j} = \alpha_j \overline{\alpha_j}$, meaning that 
	$\lambda = \beta \overline{\beta}$ for $\beta = \alpha_j/u_j$. It follows that $A = vv^*$ for $v = \beta u$. 
\end{proof}

\begin{lem}\label{lem:rank} 
If $f = \det({\rm diag}(x_1,\hdots,x_n) + \sum_{j=n+1}^m x_j A_j+A_0)$ where $A_j\in \K^{n\times n}$ are Hermitian. 
Then the rank of $A_j$ equals the degree $f$ in the variable $x_j$. 
\end{lem}

\begin{proof}
The bound $\deg_j(f)\leq {\rm rank}(A_j)$ follows from the Laplace expansion of the determinant. 
To see equality, it suffices to take $j=m=n+1$ and $A_0=0$. 
Let $f_A$ be the polynomial $f_A = \det({\rm diag}(x_1,\hdots,x_n) + A)$  where $A\in \K^{n\times n}$ is Hermitian. 
Then $f = \sum_{S\subseteq [n]}A_S{\bf x}^{[n]\backslash S}y^{|S|}$ equals the homogenization of $f_A$. 
From this we see that the degree of $f$ in the variable $y$ equals the 
size of the largest nonzero \emph{principal} minor of $A$. 
By the so-called Principal Minor Theorem \cite[Strong PMT 2.9]{KodiyalamSwan08}, 
this coincides with the size of the largest nonzero minor of $A$, i.e. ${\rm rank}(A)$. 
Therefore for a general polynomial $f = \det({\rm diag}(x_1,\hdots,x_n) + \sum_{j=n+1}^m x_j A_j+x_0A_0)$, 
the restriction to $x_k=0$ for $k\in \{n+1, \hdots, m\}\backslash \{j\}$ and $x_0=0$ has degree ${\rm rank}(A_j)$ in $x_j$, 
showing that $\deg_j(f)\geq {\rm rank}(A_j)$
\end{proof}

This immediately gives the invariance of the set of determinantal polynomials. 

\begin{cor}
	The set of polynomials in $\F[{\bf x}]_{\rm MA}$ with a determinantal representation \eqref{eq:det} is 
	invariant under the action of $\SL_2(\F)^n \rtimes S_n$. 
\end{cor}

\begin{proof} By Corollary~\ref{cor:SL2onDelta}, for any $\gamma\in \SL_2(\F)^n$, $ \Delta_{ij}(\gamma\cdot f)= \gamma\cdot \Delta_{ij}(f)$. If $\Delta_{ij}( f)$ is a Hermitian square $g\overline{g}$ with $g\in \K[{\bf x}]$ then so is $ \Delta_{ij}(\gamma\cdot f) = (\gamma\cdot g) \overline{(\gamma\cdot g)}$.
\end{proof}

\section{Determinantal Stable Polynomials}\label{sec:StableDetRep}
In this section we consider polynomials over $\R$ and $\C$ and show that 
any real stable multiaffine polynomial with a complex linear determinantal representation 
has a definite Hermitian determinantal representation (Theorem~\ref{thm:DetStableHerm}). Moreover, if the original polynomial is 
irreducible, then the matrix is diagonally similar to a Hermitian one (Theorem~\ref{thm:HermitianDiagScaling}).

We build up to the proofs of these statements with a series of useful lemmas. 

\begin{lem}\label{lem:Action2Irred2}
	Let $f\in \R[x_1, \hdots, x_m]$ be multiaffine in the variables $x_1,\hdots,x_n$ for some $n\leq m$ with coefficient of $x_1\cdots x_n$ equals to one. 
	If $f$ is irreducible, then for a generic element $\gamma \in \SL_2(\R)^n$, 
	$\partial^S(\gamma \cdot f)$ is irreducible for every $S \subset [n]$. 
\end{lem}
\begin{proof}
	For each $S\subset [n]$, the set of $\gamma\in  \SL_2(\R)^n$ for which 
	$\partial^S(\gamma\cdot f)$ is irreducible is 
	Zariski-open. Therefore it suffices to show that this set is nonempty for each $S\subset [n]$. 
	Then the intersection of these nonempty, Zariski-open sets will be nonempty and Zariski open. 
	 
	We will proceed by induction on $|S|$. For $|S| =0$, this is immediate, so 
	suppose that $|S|\geq 1$ and let $i\in S$. 
	 Note that $\partial^S(f) = \partial_i \left(\partial^{S\setminus \{i\}} f\right)$. 
	 By induction, for generic $\gamma \in \SL_2(\R)^n$, $\partial^{S\setminus \{i\}}(\gamma \cdot f)$ is irreducible. Moreover, its coefficient of $\prod_{j\in ([n]\backslash S)\cup \{i\} } x_j$ is nonzero. 
	 Therefore, up to a scalar multiple, $\partial^{S\setminus \{i\}}(\gamma \cdot f)$ satisfies the hypothesis of Lemma~\ref{lem:Action2Irred}, and hence for generic $\widetilde{\gamma} \in \SL_2(\R)$ acting on the $i$th coordinate,  
	\[
	\partial_i\left(\widetilde{\gamma} \cdot\partial^{S\setminus \{i\}}(\gamma\cdot f)\right)= 	\partial^S\left(\widetilde{\gamma} \cdot \gamma \cdot f\right) 
	\]
	is irreducible. Here we use that $\widetilde{\gamma}$ commutes with the differential operator 
	$\partial^{S\backslash \{i\}}$, since $\widetilde{\gamma}$ acts as the identity in 
	the coordinates indexed by elements of $S\backslash \{i\}$. 
	It follows that for a generic element $\gamma\in \SL_2(\R)^n$, 
	$\partial^S\left(\gamma \cdot f\right)$ is irreducible. 
\end{proof}

\begin{lem}\label{lem:PositiveCoeff}
If $g = ax_1^2+bx_1+c$ is nonnegative on $\R^m$  where  $a,b,c\in \R[x_2,\hdots,x_m]$, 
then the polynomial $a$ is nonnegative on $\R^{m-1}$. 
\end{lem}
\begin{proof}
Fix ${\bf p}\in \R^{m-1}$ and consider the specialization 
$g(x_1, {\bf p}) = a({\bf p})x_1^2+b({\bf p})x_1 + c({\bf p})$ in $ \R[x_1]$.  
Since $g$ is globally nonnegative on $\R^m$, $g(x_1, {\bf p})$ is nonnegative 
on $\R$ and so its leading coefficient $a({\bf p})$ must be nonnegative. 
\end{proof}	

\begin{lem}\label{lem:ConjFactors}
Suppose $g, h \in \C[x_1,\hdots,x_m]$ are multiaffine in $x_1, \hdots, x_n$ 
and $\partial^{[n]}g$ and $\partial^{[n]}h$ are nonzero polynomials in $x_{n+1}, \hdots, x_m$ 
of total degree at most one. If the product $g\cdot h$ has real coefficients and is nonnegative as a function on $\R^m$, then $h$ is a positive scalar multiple of $\overline{g}$, i.e. 
 $h = \lambda \overline{g}$ for some $\lambda\in \R_{>0}$. 
\end{lem}
\begin{proof}
($n=0$) Let $g= a + \ci b$ and $h = c + \ci d$ for some $a,b,c,d \in \R[x_1,\hdots,x_m]$. 
	Since $g \cdot h \in \R[x_1,\hdots,x_m]$, we see that $ad= - bc$. 
	Note that if $b=0$, then $d=0$ and so both $g$ and $h$ are real. In order for $g\cdot h$ to 
	be nonnegative on $\R^n$, we must have $h = \lambda\cdot g$ for some $\lambda\in \R_{>0}$.  
	The case $d=0$ follows similarly. 
	
	Otherwise, since $g$ and $h$ are linear and thus irreducible, either 
	$a = \lambda b$ and $c = - \lambda d$ or $a = \lambda c$ and $b = - \lambda  d$ for some 
	nonzero $\lambda \in \R$. 
	In the first case, $g=(\lambda+\ci) b$ and $h = (-\lambda+\ci)d = (\lambda-\ci)(-d)$ and thus 
	$g \cdot h = (\lambda^2 + 1)(-b\cdot d) \geq 0 $ on $\R^n$. Thus $-d= \mu b$ for some $\mu\in \R_{>0}$. 
	It follows that $h = (\lambda-\ci)(\mu b) = \mu \overline{g}$. 
	The second case gives $g = \lambda \overline{h}$. Since 
	$g\cdot h = \lambda h\cdot \overline{g}$ is nonnegative on $\R_{\geq 0}^n$, 
	we conclude $\lambda > 0$, as desired.
	
($n\geq 1$) Now suppose $n\geq 1$ and write $g = g_n x_n + g^n$ and $h = h_nx_n + h^n$. 
Since $g\cdot h$ is real and nonnegative, so is its coefficient of $x_n^2$, $g_n\cdot h_n$. 
In particular, $g_n, h_n$ satisfy the hypothesis of the theorem and so by induction, 
$h_n = \lambda \overline{g_n}$ for some $\lambda \in \R_{>0}$. 
Moreover, for every ${\bf a}\in \R^{m-1}$ with $g_n({\bf a})\neq 0$, the 
roots (in $x_n$) of the specialization of
$g\cdot h$ at ${\bf x} = {\bf a}$ come in complex conjugate pairs. 
It follows that 
$-h^n/h_n = -\overline{g^n}/\overline{g_n}$ as rational functions in $\C(x_k:k\neq n)$. 
Together with $h_n = \lambda \overline{g_n}$, this gives that $h = \lambda \overline{g}$. 
Moreover, since $g\cdot h = \lambda\cdot  g\cdot \overline{g}$ is nonnegative on $\R^n$, we see that $\lambda >0$. 
\end{proof}

\begin{thm}\label{thm:DetStableHerm}
	Let $f \in \R[x_1,\hdots,x_m]$  be stable and complex determinantal, i.e.
	\[
	f = \det\left(\diag(x_1,\hdots,x_n)+ \sum_{j=n+1}^m A_j x_j+A_0 \right)
	\]
	for some $n\times n$ complex matrices $A_j$. 
	Then there exists Hermitian matrices $B_0, B_{n+1}, \hdots, B_m$ 
	for which $f = \det\left(\diag(x_1,\hdots,x_n)+ \sum_{j=n+1}^m B_j x_j+B_0 \right).$
\end{thm}
\begin{proof} 
First suppose $f$ is irreducible. 
By Lemma \ref{lem:Action2Irred2}, there is $\gamma \in \SL_2(\R)^n$, such that $\partial^{S}(\gamma \cdot f)$ is irreducible for all $S\subset [n]$. 
	By Corollary~\ref{cor:Invariance}, we can replace $f$ by $\gamma\cdot f$, 
	and thereby assume that all the coefficients of $\prod_{k \in [n]\setminus\{i,j\}}x_k^2$ 
	in the polynomials $\Delta_{ij}(\gamma \cdot f)$ are non-zero.  To see this, notice that by induction on $n$, we can prove that
	\[
	{\rm coeff}\left(\Delta_{ij}(f), \prod_{k \in [n]\setminus\{i,j\}}x_k^2\right) = \Delta_{i,j}\left(\partial^{[n]\setminus \{i,j\}}( f)\right).
	\]
	If this coefficient is zero, then Lemma~\ref{lem:factorDelta} implies that $\partial^{[n]\setminus \{i,j\}}( f)$ is reducible.

Let $i<j\in [n]$. 
Since $f$ is determinantal, by Theorem~\ref{thm:DeltaDetRep}, the polynomial $\Delta_{ij}(f)$ factors as $g_{ij}\cdot g_{ji}$ where $g_{ij},g_{ji}$ are multiaffine in $\{x_k : k\in [n]\backslash\{i,j\}\}$ and has total degree $\leq n-1$. 
In particular, the coefficient of $\prod_{k \in [n]\setminus\{i,j\}}x_k$ in both $g_{ij}$ 
and $g_{ji}$ has degree $\leq 1$ in $x_{n+1}, \hdots, x_m$. 
By the arguments above we can assume this coefficient is nonzero.  Since $f$ is real stable, $\Delta_{ij}(f)$ is also globally nonnegative on $\R^n$ \cite{Branden07}.  Therefore by Lemma~\ref{lem:ConjFactors}, $g_{ji} = \lambda \overline{g_{ij}}$ for some $g_{ij}$.  It follows that $\Delta_{ij}(f)$ factors as a Hermitian square $h_{ij}\cdot \overline{h_{ij}}$ where 
$h_{ij} = \sqrt{\lambda} g_{ij}$.  
Theorem~\ref{thm:HerDetRep} then gives the desired Hermitian determinantal representation. 

Now suppose $f$ is reducible, say $f = f_1 \cdots f_r$ where each factor $f_k$ 
is irreducible and multiaffine in the variables $x_i$ for $i\in I_k\subset [n]$. 
Each factor is stable. Moreover, by 
Lemma~\ref{lem:factorDelta}, $\Delta_{ij}(f_k)$ is either zero or factors as a product 
of two polynomials that are multiaffine in $\{x_{\ell}:{\ell}\in I_k\} $ and 
with total degree $\leq |I_k|-1$. Since $f_k$ is irreducible, 
the arguments above show that for every $i,j\in I_k$, $\Delta_{ij}(f_k)$ is a Hermitian square, 
from which it follows that $\Delta_{ij}(f) = \Delta_{ij}(f_k)\cdot \prod_{\ell\neq k}f_{\ell}^2$ is a Hermitian square. Theorem~\ref{thm:HerDetRep} then gives the desired Hermitian determinantal representation. 
\end{proof}
\begin{remark}
Theorem~\ref{thm:DetStableHerm} cannot hold for arbitrary real stable polynomials. 
For example, consider $f$ to be the basis generating polynomial of the V\'amos matriod, defined in 
\cite{BrandenDetRep}. It was shown by Wagner and Wei \cite{WW09} that $f$ is stable. 
By the theory of matrix factorizations, some power $f^r$ of $f$ has a complex linear determinantal representation 
(see \cite[\S 3.3]{VinnikovSurvey}).  This power is necessarily stable, but as shown by Br\"and\'en \cite{BrandenDetRep}, $f^r$ does not have a definite Hermitian determinantal representation. 
\end{remark}
When $f$ is reducible, one can easily construct determinantal representations of $f$ that are not Hermitian by taking block upper triangular representations.  
For example, $x_1x_2$ equals $\det\begin{pmatrix} x_1 & 1 \\ 0 & x_2\end{pmatrix}$. 
However, when $f$ is irreducible and real stable, we see that all complex linear determinantal representations 
are Hermitian, up to conjugation by diagonal matrices. 

\begin{thm}\label{thm:HermitianDiagScaling}
	Let $f \in \R[x_1,\hdots,x_m]$  be stable, irreducible, and complex determinantal, i.e.
	\[
	f = \det\left(\diag(x_1,\hdots,x_n)+ \sum_{j=n+1}^m A_j x_j+A_0 \right)
	\]
	for some $n\times n$ complex matrices $A_j$. 
	Then there exists a real diagonal matrix $D \in \R^{n\times n}$ such that $D^{-1} A_j D$ is Hermitian for all $j$.
\end{thm}
\begin{proof}
By Lemma \ref{lem:Action2Irred2}, there is $\gamma \in \SL_2(\R)^n$, such that $\partial^{S}(\gamma \cdot f)$ is irreducible for all $S\subset [n]$. 
	By Corollary~\ref{cor:Invariance}, we can replace $f$ by $\gamma\cdot f$, 
	and thereby assume that all the coefficients of $\prod_{k \in [n]\setminus\{i,j\}}x_k^2$ 
	in the polynomials $\Delta_{ij}(\gamma \cdot f)$ are non-zero, as in the proof of 
	Theorem~\ref{thm:DetStableHerm}.	
	
	Let $A(x) = \sum_{k=n+1}^m A_k x_k + A_0$ and let $a_{ij}\in \C[x_{n+1},\hdots, x_m]$ 
	denote the $(i,j)$th entry of $A(x)$. Then the coefficient of $\prod_{k\in [n]\setminus\{i,j\}} x_k^2$
	in $\Delta_{ij}f$  is $a_{ij} a_{ji}$.
Since $f$ is stable, the polynomial $\Delta_{ij}(f)$ is nonnegative on $\R^m$. 
	Then by Lemma~\ref{lem:PositiveCoeff}, it follows that the coefficient $a_{ij} a_{ji}$ of 
	$\prod_{k\in [n]\setminus\{i,j\}}x_k^2$ in $\Delta_{ij}(f)$ is nonnegative on $\R^{n-m}$. 
	By Lemma~\ref{lem:ConjFactors}, we can conclude that for each $1\leq i<j\leq n$, 
	there is some $\lambda_{ij}\in \R_{>0}$ such that $a_{ij} = \lambda_{ij} \overline{a_{ji}}$. 
	
	We claim that the scalars $\lambda_{ij}$ satisfy 
	$\lambda_{ij} = \lambda_{ik}\lambda_{kj}$ for all $1\leq i<k<j\leq n$. 
	For simplicity, we show this for $i=1$, $k=2$, $j=3$ and the proof in general is identical. 
	By the arguments above, the starting determinantal representation of $f$ has the form 
\[
{\rm diag}(x_1, \hdots, x_n) + A(x) = 
\begin{pmatrix}
x_1 +a_{11} 				& a_{12} & a_{13} & \hdots  & a_{1n}\\
\lambda_{12} \overline{a_{12}} & x_2 + a_{22} & a_{23} &\hdots  & \\
\lambda_{13} \overline{a_{13}} & \lambda_{23} \overline{a_{23}}& x_3 + a_{33} &  & \\ 
\vdots & \vdots  & &\ddots   &  \\
\lambda_{1n} \overline{a_{1n}} && & & x_n+a_{nn} 
\end{pmatrix}.
\]	
Recall that by Dodgson condensation, the polynomial $\Delta_{ij}(f)$ 
factors as $\det(M[i,j])\cdot \det(M[j,i])$ where $M[i,j]$ is the matrix obtained from $M = {\rm diag}(x_1, \hdots, x_n)+A(x)$ 
by removing the $i$th row and $j$th column. 
These polynomials are affine in $x_k$ for $k\in [n]\backslash \{i,j\}$. 
In particular, 
\begin{align*}
g := \partial^{[n] \setminus\{1,2,3\}}\det(M[3,1]) &=a_{12}a_{23} - a_{13} (x_2 + a_{22}), \text{ and }  \\
h := \partial^{[n] \setminus\{1,2,3\}}\det(M[1,3]) &= \lambda_{12} \lambda_{23}\overline{a_{12}}  \overline{a_{23}} -  \lambda_{13} \overline{a_{13}}(x_2 + a_{22}).
\end{align*}

These polynomials satisfy the hypotheses of Lemma~\ref{lem:ConjFactors}, and so 
there is some $\mu\in \R_{>0}$ for which $h = \mu \overline{g}$. 
Since $a_{ij}$ is nonzero for all $i,j$ and $a_{22}$ is invariant under conjugation, 
we see that $\lambda_{12} \lambda_{23} = \mu = \lambda_{13}$. 
More generally $\lambda_{ij} = \lambda_{ik}\lambda_{kj}$ for any $i<k<j$. 

Now define $D = {\diag}(1,\sqrt{\lambda_{12}}, \hdots, \sqrt{\lambda_{1n}})$. 
For $i<j$, $\lambda_{1j}=\lambda_{1i}\lambda_{ij}$
we calculate the $(i,j)$th and $(j,i)$th entries of $D^{-1} A(x) D$ as
\[
(D^{-1} A(x) D)_{ij} = 
\frac{\sqrt{\lambda_{1j}}}{\sqrt{\lambda_{1i}}} a_{ij} 
= 
\sqrt{\lambda_{ij}}a_{ij}
\ \ \ \text{ and } \ \ \
(D^{-1} A(x) D)_{ji} = 
\frac{\sqrt{\lambda_{1i}}}{\sqrt{\lambda_{1j}}}\lambda_{ij}\overline{a_{ij}} 
= 
\sqrt{\lambda_{ij}}\overline{a_{ij}}.
\] 
\end{proof}

%%%%%%%%%%%%%%%%%%%%%%%%%%%%%%%%%%%%%%%%%%%%%%%%%%%%%%%%
%%%%%%%%%%%%%%%%%%%%%%%%%%%%%%%%%%%%%%%%%%%%%%%%%%%%%%%%

\section{Defining the set of factoring multiquadratic polynomials and the image of the principal minor map}\label{sec:factoringMQ}

In this section we give a complete characterization of the image of the principal minor map of Hermitian matrices 
using the characterization of Hermitian multiaffine determinantal polynomials from Section~\ref{sec:otherDetRep} 
and the characterization of multiquadratic polynomials that are Hermitian squares. 
This set is invariant under the action of $\SL_2(R)^n \rtimes S_n$ and we derive 
the defining equations and numerical conditions  as the orbit of a finite set under the action of this group, where
 $R$ is a unique factorization domain. In this section, we will restrict to rings and fields of characteristic 
$\neq 2$.

\begin{lem}\label{lem:UnivariateQuadraticFactor}
	Let $g = ax^2 + bx +c\in R[x]$. 
	The polynomial $g$ factors in to two linear factors in $R[x]$
	if and only if its discriminant ${\rm Discr}_{x}(g)$ is a square in $R$. 
\end{lem}
\begin{proof}
	($\Rightarrow$) If $g$ factors, then it has a root in the fraction field of $R$. 
	By the quadratic formula, this implies that the discriminant is a 
	square in ${\rm frac}(R)$, and hence in $R$.  
	
	($\Leftarrow$) Suppose that $b^2 - 4ac = q^2$ for some $q\in R$. 
	We can rewrite this as $(b-q)(b+q) = 4ac$. 
	Since $R$ is a unique factorization domain, there is some choice of factorization 
	of $a = a_1a_2$ and $c = c_1c_2$ so that $b-q = 2a_1c_1$ and $b+q = 2a_2c_2$. 
	If $a = 0$, then $g$ factors as $1\cdot (bx+c)$, so we can assume $a\neq 0$. 
	We can then write $g$ as
	\[
	g = a\left(x - \frac{-b+q}{2a}\right)\left(x - \frac{-b-q}{2a}\right) = a_1a_2\left(x + \frac{c_1}{a_2}\right)\left(x + \frac{c_2}{a_1}\right)
	= (a_2x + c_1)(a_1x + c_2).
	\]
\end{proof}

This lemma does not hold over rings of characteristic two. 
See \cite[Section~2.4, Exercise~6]{CoxGalois} for further discussion. 
Note that for $g\in R[x,y]_{\rm MQ}$, ${\rm Discr}_x(g)$ is a polynomial of degree $4$ in $y$ whose coefficients are quadratic in the coefficients of $g$. 

\begin{lem}\label{lem:squareCubics} Let $h(x) = \sum_{i=0}^4 b_i x^i \in R[x]_{4}$ a univariate quartic.
	Then $h$ is a square in $R[x]$ if and only if $b_0$, $b_4$ and $h(1) = \sum_j b_j$ are squares in $R$ and the point $(b_0, b_1, b_2, b_3, b_4)$ satisfies
	\begin{align}\label{equ:mainequs}
		& \ \ \ \  b_4 b_1^2 - b_3^2 b_0 = 0,   \ \  \ \ \ \ b_3^3 - 4 b_4b_3b_2 + 8b_4^2b_1 = 0, \ \ \ \ \ b_1^3 - 4 b_0b_1b_2 + 8b_0^2b_3 = 0\\ \nonumber
		&   b_2 b_3^2 - 4 b_2^2 b_4 + 2 b_1 b_3 b_4 + 16 b_0 b_4^2=0, \ \ \text{ and } \ \   b_1^2 b_2 - 4 b_0 b_2^2 + 2 b_0 b_1 b_3 + 16 b_0^2 b_4=0. 
	\end{align}
\end{lem}

\begin{proof}
	($\Rightarrow$) If $h(x)$ is a square in $R[x]$, then 
	$h(x) = \sum_{i=0}^4 b_i x^i = (\alpha x^2 + \beta x + \delta)^2$ for some $\alpha, \beta, \delta\in R$. We see that  $b_4 = \alpha^2$, $b_0 = \delta^2$, and $\sum_{i=0}^4 b_i = (\alpha+\beta+\delta)^2$ are all  squares in $R$. 
	Each of the coefficients $b_i$ is a polynomial in $\alpha$, $\beta$, $\delta$ and one
	can quickly check that all the cubics in $(\ref{equ:mainequs})$ vanish identically on 
	this parametrization. 
	
	($\Leftarrow$) 
	Let $b_4 = \alpha^2$, $b_0 = \delta^2$, and $\sum_jb_j = \lambda^2$
	for some  $\alpha$, $\delta$, $\lambda \in R$. 
	From $b_0b_3^2=b_1^2b_4$, we see that $\delta b_3 = \pm b_1\alpha$, and  
	replacing $\alpha$ with $-\alpha$ if necessary, we can take $\delta b_3 = b_1 \alpha$. 
	
	If $b_3$ is nonzero, we see from the second equation that 
	$b_4$, and hence $\alpha$, must also be nonzero. Define $\beta = b_3/(2\alpha)\in {\rm frac}(R)$. 
	It follows immediately 
	that $b_0 = \delta^2$, $b_1 = 2 \delta\beta$,  $b_3 = 2 \beta\alpha$, and  $b_4 = \alpha^2$.
	If $b_3 \neq 0$,  the second equation implies that  
	\[b_2 = \frac{1}{4b_3b_4}(b_3^3 + 8b_1b_4^2) 
	=  \frac{1}{8\beta\alpha^3}(8\beta^3\alpha^3 + 16\delta\beta\alpha^4) = \beta^2 + 2 \delta\alpha,\]
	from which we conclude that $(\alpha x^2 + \beta x + \delta)^2 = h(x)$. 	
	Similarly, if $b_1$ is nonzero then so are $b_0$ and $\delta$. We can define $\beta = b_1/(2\delta) $ and use $4b_0b_1b_2 = b_1^3+ 8b_0^2b_3$  to conclude that $(\alpha x^2 + \beta x + \delta)^2 = h(x)$.	
	In either case, evaluating at $x=1$ gives that $\alpha + \beta + \delta = \pm \lambda$, 
	and $\beta =  \pm \lambda - \alpha - \delta \in R$. 
	
	If $b_1 = b_3 = 0$, the equations simplify to 
	$4b_4(b_2^2- 4b_0b_2)=0$ and  $4b_0(b_2^2 - 4b_0b_2)=0$. 
	If $b_0$ or $b_4$ is nonzero, then $b_2 = \pm \delta \alpha$ and $h(x)$ is 
	$(\alpha x^2 \pm \delta)^2$.  Otherwise  
	$b_0=b_1 = b_3 = b_4 = 0$, in which case $b_2 = \lambda^2$ and  $h(x) = (\lambda x)^2$. 
\end{proof}

\begin{cor}\label{cor:quarticSquare} 
	A quartic $h(x) = \sum_{j=0}^4 b_j x^j$ 
	is a square in $R[x]$ if and only if for all $\gamma$ in $ {\rm SL}_2(\{0,\pm1\})$,  
	$(\gamma\cdot h)_{x=0}$ is a square in $R$ and 
	$B_x(\gamma\cdot h) = C_x(\gamma\cdot h)=D_x(\gamma\cdot h) =0$, where 
	\[
	B_x(h) = b_4 b_1^2 - b_3^2 b_0,  \ \
	C_x(h)= b_1^3 - 4 b_0b_1b_2 + 8b_0^2b_3,   \text{ and }
	D_x(h) =b_1^2 b_2 - 4 b_0 b_2^2 + 2 b_0 b_1 b_3 + 16 b_0^2 b_4.
	\]
\end{cor}

\begin{proof}
	It suffices to show that we can recover the conditions in Lemma~\ref{lem:squareCubics}, 
	which we can do this with three elements of $\SL_2(\{0,\pm1\})$: the identity, $\gamma_1 = {\small \begin{pmatrix} 0 & 1 \\ -1 & 0 \end{pmatrix}}$ and 
	$\gamma_2 = {\small \begin{pmatrix} 1 & 1 \\ 0 & 1 \end{pmatrix}}$, representing the 
	fractional linear transformations $x\mapsto -1/x$ and $x \mapsto x+1$, respectively. 
	Note that $(\gamma_1\cdot h)(0) = b_4$ and $(\gamma_2\cdot h)(0) = h(1) = \sum_j b_j$,
	so from (i), we recover that all of these are squares in $R$. 
	The element $\gamma_1$ induces the transposition $b_{k}\mapsto (-1)^kb_{4-k}$ for each $k$.
	One can quickly check that we recover the two missing cubics from this action of $\gamma_1$. 
\end{proof}

\begin{remark}
	The ideal generated by the five cubics in Lemma~\ref{lem:squareCubics} is not saturated with respect to 
	the ideal $\langle b_0, \hdots, b_4\rangle$.  Its saturation is minimally generated 
	by these five cubics together with  $-b_1b_3^2+4b_1b_2b_4-8b_0b_3b_4$ and $-b_1^2b_3+4b_0b_2b_3-8b_0b_1b_4$.
\end{remark}

Note that the coefficients of ${\rm Discr}_y(\gamma\cdot g)$ have degree two 
in the coefficients $c_\alpha$ of $g$, and so the polynomials listed in (ii) above 
have degree six.  For example,
\begin{align*}
	B_x({\rm Discr}_y(\gamma\cdot g)) = & \ 4(c_{01}c_{11} - 2(c_{10}c_{02} + c_{00}c_{12}))^2(c_{21}^2 - 4c_{20}c_{22})\\
										& - 4(c_{01}^2 - 4c_{00}c_{02})(c_{11}c_{21} - 2(c_{20}c_{12} + c_{10}c_{22}))^2.
\end{align*}

\begin{thm}\label{thm:prod}
	Let $R$ be a unique factorization domain with ${\rm char}(R)\neq 2$ and $|R|\geq 13$. 
	A polynomial $g = \sum_{\alpha \in \{0,1,2\}^n}c_{\alpha}{\bf x}^{\alpha}\in R[{\bf x}]$
	is the product of multiaffine polynomials if and only 
	if for all $\gamma\in \SL_2(R)^n \rtimes S_n$, 
	\begin{itemize}
		\item[(i)] ${\rm Discr}_{x_1}(\gamma\cdot g)|_{x_2 = \hdots = x_n =0} 
		= \gamma\cdot (c_{1{\bf 0}}^2 - 4c_{0{\bf 0}}c_{2{\bf 0}}) $ is a square in $R$,  
		\item[(ii)] the sextic polynomials in ${\bf c}$ given by specializing 
		$B_{x_2}({\rm Discr}_{x_1}(\gamma\cdot g))$, $C_{x_2}({\rm Discr}_{x_1}(\gamma\cdot g))$
		and $D_{x_2}({\rm Discr}_{x_1}(\gamma\cdot g))$ to $x_3 = \hdots = x_n =0$ are all zero. 
	\end{itemize}
\end{thm}

\begin{proof}
	We can express $g = \sum_{\beta\in\{0,1,2\}^2}g_{\beta}x_1^{\beta_1}x_2^{\beta_2}$ 
	where $g_{\beta}\in R[x_3, \hdots, x_n]_{\leq {\bf 2}}$. 
	The polynomial $B_{x_2}({\rm Discr}_{x_1}(g))$ 
	has degree six in the coefficients $g_{\beta}$ and so degree $\leq 12$ in each variable $x_j$. 
	
	Consider $I\subset R$ with $|I|=13$. 
	For $\lambda_1 = \lambda_2 = 0$ and $\lambda_3, \hdots, \lambda_n\in I$, 
	consider the element 
	$\gamma = \left({\tiny \begin{pmatrix} 1 & \lambda_j \\ 0 & 1 \end{pmatrix}}\right)_{j}$
	in $\SL_2(R)^n$. 
	For any polynomial $F\in R[{\bf x}]$ the evaluation of 
	$\gamma\cdot F$ at ${\bf x} = 0$ equals the evaluation of $F$ 
	at ${\bf x} = (\lambda_1, \hdots, \lambda_n)$. 
	In particular, (ii) implies that the polynomials 
	$B_{x_2}({\rm Discr}_{x_1}(\gamma \cdot g))$, $C_{x_2}({\rm Discr}_{x_1}(\gamma\cdot g))$
	and $D_{x_2}({\rm Discr}_{x_1}(\gamma\cdot g))$ vanish at the point 
	${\bf x} = (\lambda_1, \hdots, \lambda_n)$ for every choice of $\lambda_j\in I$. 
	Since these polynomials have degree $\leq 12$ in each variable $x_j$ and $|I|\geq 13$, 
	it follows that each of these polynomials is identically zero, using \cite[Lemma 4.1]{AV21}. 
	
	We can now proceed by induction on $n$. The $n=1$ case is the content of 
	Lemma~\ref{lem:UnivariateQuadraticFactor} together with the observation that the 
	discriminant is invariant under the action of $\SL_2(R)$, so we suppose $n\geq 2$. 
	Let $h = {\rm Discr}_{x_1}(g) \in S[x_2]$ where $S = R[x_3, \hdots, x_n]$. 
	By induction, 
	for every $\gamma\in \SL_2(R)$ acting on the variable $x_2$, 
	$(\gamma \cdot g)|_{x_2=0}$ factors into multiaffine polynomials 
	and so $(\gamma \cdot h)|_{x_2=0} = {\rm Discr}_{x_1}((\gamma \cdot g)|_{x_2=0})$ 
	is a square in $S = R[x_3, \hdots, x_n]$.  
	
	By Corollary~\ref{cor:quarticSquare}, it follows that ${\rm Discr}_{x_1}(g)$ 
	is a square in $S[x_2]$. 
	Then by Lemma~\ref{lem:UnivariateQuadraticFactor}, 
	$g$ factors into linear factors in $x_1$ in the ring $S[x_1,x_2] = R[{\bf x}]$. 
	Using the action of $S_n$, we see that every irreducible factor of $g$ must have 
	degree $\leq 1$ in each variable. 
\end{proof}

\begin{remark}\label{rem:DimensionOfEquations}
	For every choice of $i\neq j\in [n]$ and ${\bf \lambda} \in I^{n-3}$ we 
	obtain three equations by evaluating 
	$B_{x_j}({\rm Discr}_{x_i}(g))$, $C_{x_j}({\rm Discr}_{x_i}(g))$
	and $D_{x_j}({\rm Discr}_{x_i}(g))$ at the point $\lambda$, along with 
	additional two polynomials from the two missing analogous polynomials in Lemma~\ref{lem:squareCubics}, 
	which can be recovered from the ${\rm SL}_2$-action on $x_j$. 
	This gives a total of $5n(n-1)13^{n-3}$ sextic equations in the coefficients of $g$. 
\end{remark}

\begin{lem}\label{lem:UnivariateConjFactor}
	Let $S$ be a unique factorization domain  with ${\rm char}(S)\neq 2$ and an automorphic involution $a \rightarrow \overline{a}$ and let $R$ be the fixed ring under this involution.
	The polynomial $g = ax^2 + bx +c\in R[x]$ is a Hermitian square in $S[x]$
	if and only if $a$ and $c$ are Hermitian squares in $S$ and the discriminant ${\rm Discr}_{x}(g)=q^2$ with $q \in S[x]$ and $\overline{q} = - q$.
\end{lem}
\begin{proof}
	($\Rightarrow$) If $g$ factors into two conjugates $(sx + t)(\overline{s} x + \overline{t})$, then $a = s \overline{s}$ and $c = t \overline{t}$ and 
	\[
	{\rm Disc}_x(g) = b^2 - 4 a c = (s \overline{t} + t \overline{s})^2 - 4 s \overline{s} t \overline{t} = (s \overline{t} - t \overline{s})^2
	\]
	which satisfies the desired property.
	
	($\Leftarrow$) Assume that $b^2 - 4ac = q^2$ such that $\overline{q} = -q$. If $a = 0$, then $b = \pm q$ and thus $\overline{b} = -b$. Since $b \in R$, 
	then $b=0$ and $g = c$ is a Hermitian square as desired. If $a \neq 0$, 
	then $(b-q)(b+q) = (b-q)(b-\overline{q}) = 4ac = 4 s \overline{s} t \overline{t}$, 
	where $a = s \overline{s}$ and $c = t \overline{t}$. Thus, after relabeling if needed, we may assume that $b-\overline{q} = 2st$. 
	Thus, we can write $g$ as
	\[
	g = a\left(x - \frac{-b+\overline{q}}{2a}\right)\left(x - \frac{-b+q}{2a}\right) = s \overline{s}\left(x + \frac{t}{\overline{s}}\right)\left(x + \frac{\overline{t}}{s}\right)
	= (sx + t)(\overline{s}x + \overline{t}).
	\]
\end{proof}

\begin{thm}\label{thm:conj}
	Let $S$ be a unique factorization domain with ${\rm char}(S)\neq 2$ and an automorphic involution $a\mapsto \overline{a}$.  Let $R$ be the fixed ring of this automorphism with $|R|\geq 13$. The polynomial $g = \sum_{\alpha \in \{0,1,2\}^n}c_{\alpha}{\bf x}^{\alpha}$ in $  R[{\bf x}]_{\rm MQ}$ is a Hermitian square if and only if $(\gamma \cdot g)|_{x_3=\hdots = x_n = 0}$ is a Hermitian square in $ S[x_1, x_2]$ for all $\gamma \in \SL_2(R)^n\rtimes S_n$.	
\end{thm}
\begin{proof}
	If for all $\gamma \in \SL_2(R)^n\rtimes S_n$, the polynomial $(\gamma \cdot g)|_{x_3=\hdots = x_n = 0}$ is a Hermitian square in $S[x_1,x_2]$, then by Lemma~\ref{lem:UnivariateQuadraticFactor}, ${\rm Discr}_{x_1}(\gamma\cdot g)|_{x_3=\hdots = x_n = 0} $ is a square in $S[x_2]$ . Using Corollary~\ref{cor:quarticSquare} we see that the two conditions of Theorem~\ref{thm:prod} are satisfied and hence  we deduce that $g$ is a product of multiaffine polynomials in $S[{\bf x}]$. 
	To prove that $g$ is a Hermitian square, we will proceed by induction on $n$. The case $n=2$ is trivially satisfied. For the inductive step, write $g$ as $g = p_2 x_1^2 + p_1 x_1 + p_0$ for some $p_2,p_1,p_0 \in \tilde{R} = R[x_2,\hdots,x_{n}]$. By induction we see that $p_2$ and $p_0$ are both Hermitian squares and as $g$ is a product of multiaffine polynomials, then by Lemma~\ref{lem:UnivariateQuadraticFactor}, we see that ${\rm Disc}_{x_1}(g) = p_1^2 - 4p_2 p_0 = q^2$
	for some $q\in S[x_2,\hdots,x_{n}]$. Since $p_1^2 - 4 p_2 p_0 \in \tilde{R}$, then $q^2 \in \tilde{R}$ and so $q= -\overline{q}$ or $q = \overline{q}$. In the former case, Lemma~\ref{lem:UnivariateConjFactor} implies that $g$ is a Hermitian square and we are done. Otherwise we get $(\gamma \cdot q)|_{{\bf x} = {\bf 0}}= \overline{(\gamma \cdot q)|_{{\bf x} = {\bf 0}}}$ for all $\gamma \in \SL_2(R)^{n-1}$. Notice that by induction on the other hand, $(\gamma\cdot g)|_{{\bf x} = {\bf 0}}$ is a Hermitian square and hence 
	\[
	{\rm Disc}_{x_1}((\gamma \cdot g)_{{\bf x} = {\bf 0}}) = \left(\gamma\cdot(p_1^2 - 4p_2 p_0)\right)_{{\bf x} = {\bf 0}} = (\gamma \cdot q)_{{\bf x} = {\bf 0}}^2 	\text{ with } (\gamma \cdot q)|_{{\bf x} = {\bf 0}}=- \overline{(\gamma \cdot q)|_{{\bf x} = {\bf 0}}}.
	\] 
     Thus we conclude that $(\gamma \cdot q)|_{{\bf x} = {\bf 0}}=0$ for all $\gamma \in \SL_2(R)^{n-1}$. Consider $\gamma = (\gamma_i)_{2\leq i \leq n}$ where $\gamma_i = \tiny\begin{pmatrix}
		1 & \lambda_i \\ 0 & 1
	\end{pmatrix}$ for $\lambda_i \in R$. Notice that $\gamma\cdot q|_{(x_2 = \cdots = x_n = 0)} = q|_{(x_2 = \lambda_2, \hdots,x_n = \lambda_{n})}=0$. Since $|R|\geq 3$, \cite[Lemma~4.1]{AV21} implies that $q \equiv 0$ and thus $q=-\overline{q}$ and we apply Lemma~\ref{lem:UnivariateConjFactor} again to deduce that $g$ is a Hermitian square.
	\end{proof}
	Let $\F$ be a field of ${\rm char}(\F)\neq 2$ with $|\F|\geq 13$ and $\K$ be a degree two extension field. 
	Let $\delta$ denote the square root of the discriminant of the minimal polynomial of this field extension. 
	Then $\K = \F(\delta)$ and the involution $\delta\longrightarrow \overline{\delta} = -\delta$ extends to an automorphism of $\K$ with 
	fixed field $\F$.	
	
	\begin{remark}\label{rk:pureImg}
		In the field $\K$, $\overline{q} = -q$ is equivalent to requiring $q = \delta r$ for some $r \in \F$.
	\end{remark}
	
	\begin{lem}\label{lem:TwoVarConjFactor}
		Let $g = \sum_{\alpha \in \{0,1,2\}^2}c_{\alpha}{\bf x}^{\alpha}\in \F[x_1,x_2]_{MQ}$. The polynomial $g$ is a Hermitian square in $\K[x_1,x_2]$
		if and only if for all $\gamma \in \SL_2(\{0,\pm 1\})^2\rtimes S_2$:
		\begin{itemize}
			\item[(i)] $\gamma \cdot c_{(0,0)}$ is a Hermitian square in $\K$.
			\item[(ii)] $\frac{1}{\delta^2}{\rm Discr}_{x_1}(\gamma\cdot g)$ is a square in $\F[x_2]$.
			%\item[(ii)] $\frac{1}{\delta^2}{\rm Discr}_{x_1}(\gamma\cdot g)|_{x_2 =0} 
			%= \gamma\cdot \left(\frac{1}{\delta^2}(c_{1{0}}^2 - 4c_{0{ 0}}c_{2{ 0}})\right) $ is a square in $\F$,
			%\item[(iii)] the sextic polynomials in ${\bf c}$ given by  
			%$B_{x_2}\left({\rm Discr}_{x_1}(\gamma\cdot g)\right)$, $C_{x_2}\left({\rm Discr}_{x_1}(\gamma\cdot g)\right)$ and
			%  $D_{x_2}\left({\rm Discr}_{x_1}(\gamma\cdot g)\right)$ are all zero.   
		\end{itemize}
	\end{lem}
	\begin{proof}
		Write $g$ as $g = p_2 x_1^2 + p_1 x_1 + p_0$ where $p_2,p_1$ and $p_0$ are quadratics in $\F[x_2]$. Using Lemma~\ref{lem:UnivariateConjFactor}, we see that $g$ is a product of two conjugate factors if and only if $p_2$ and $p_0$ are product of two conjugates in $\K[x_2]$ and ${\rm Disc}_{x_1}g = q^2$ where $\overline{q} = -q$ for some $q \in \K[ x_2]$. Notice that by Remark~\ref{rk:pureImg}, this condition is equivalent to $q = \delta r$ where $r \in \F[{x_2}]$ and thus requiring that $\frac{1}{\delta^2}{\rm Disc}_{x_1}g$ is a square in $\F[x_2]$. Using Lemma~\ref{lem:UnivariateConjFactor},  $p_2$ and $p_0$ are conjugates if and only if $c_{(i,j)}$ is a product of two conjugates for $i,j\in \{0,2\}$ and $\frac{1}{\delta^2}{\rm Discr}_{x_2}(\gamma\cdot g)|_{x_1 =0}$ is a square for $\gamma \in \SL_2(\F)$ and this gives the desired equivalence.
	\end{proof}

	\begin{thm}\label{thm:Conj}
		A polynomial $g = \sum_{\alpha \in \{0,1,2\}^n}c_{\alpha}{\bf x}^{\alpha}\in \F[{\bf x}]$
		is a Hermitian square in $\K[{\bf x}]$ if and only 
		if for all $\gamma\in \SL_2(\F)^n \rtimes S_n$, 
		\begin{itemize}
			\item[(i)] $(\gamma\cdot c_{\bf 0})$ is a Hermitian square in $\K$,
			%\item[(ii)] $\frac{1}{\delta^2}{\rm Discr}_{x_1}(\gamma\cdot g)$ is a square in $\F[x_2,\hdots,x_n]$.
			\item[(ii)] $\frac{1}{\delta^2}{\rm Discr}_{x_1}(\gamma\cdot g)|_{x_2 = \hdots = x_n =0} 
			= \gamma\cdot \left(\frac{1}{\delta^2}(c_{1{\bf 0}}^2 - 4c_{0{\bf 0}}c_{2{\bf 0}})\right) $ is a square in $\F$,  
			\item[(iii)] the sextic polynomials in ${\bf c}$ given by specializing 
			$B_{x_2}\left({\rm Discr}_{x_1}(\gamma\cdot g)\right)$, $C_{x_2}\left({\rm Discr}_{x_1}(\gamma\cdot g)\right)$
			and $D_{x_2}\left({\rm Discr}_{x_1}(\gamma\cdot g)\right)$ to $x_3 = \hdots = x_n =0$ are all zero. 
		\end{itemize}
	\end{thm}
	\begin{proof}
		Using Lemma~\ref{thm:conj}, $g$ is a Hermitian square in $\K[{\bf x}]$ if and only if for all $\gamma \in \SL_2(\F)^n \rtimes S_n$, $(\gamma \cdot g)|_{x_3=\hdots = x_n = 0}$ is a Hermitian square in $\K[x_1,x_2]$. Now Lemma~\ref{lem:TwoVarConjFactor}, shows that this  is equivalent to $\gamma \cdot c_{\bf 0}$ is a product of two conjugates and $\frac{1}{\delta^2}{\rm Discr}_{x_1}(\gamma\cdot g)$ is a square in $\F[x_2]$, which is equivalent to conditions (ii) and (iii) above using Corollary~\ref{cor:quarticSquare}.
	\end{proof}
	
	Now we are ready to give a complete characterization of the image of the principal minor map of Hermitian matrices using the 
	characterization of Hermitian multiaffine determinantal polynomials from Section~\ref{sec:otherDetRep} 
	and the characterization of multiquadratic polynomials that are Hermitian squares.
		
	Recall that to each element ${\bf a} = (a_S)_{S\subseteq [n]}$ in $\F^{2^n}$ we associate the multiaffine polynomial 
	\[f_{\bf a} = \sum_{S\subseteq [n]}a_S {\bf x}^{[n]\backslash S}.\] 	
	For $n=3$, the discriminant of the Rayleigh difference $\Delta_{12}(f)$ with respect to 
	$x_3$ is Cayley's $2\times 2\times 2$ hyperdeterminant
\begin{align*}
\HypDet({\bf a}) 
= &  \ ( a_{1} a_{23} + a_{2} a_{13}-a_{3} a_{12} - a_{\emptyset} a_{123})^2 -  4(a_1a_2 - a_{\emptyset} a_{12}) (a_{13}a_{23} - a_3a_{123}) \\ 
=& \ a_{\emptyset}^2 a_{123}^2 +  a_{1}^2 a_{23}^2 +a_{2}^2 a_{13}^2+a_{3}^2 a_{12}^2  - 2 a_{\emptyset} a_{1} a_{23} a_{123} - 2 a_{\emptyset} a_{2} a_{13} a_{123} -2 a_{\emptyset} a_{3} a_{12} a_{123}\\
 &-2 a_{1} a_{2} a_{13} a_{23} - 2 a_{1} a_{3} a_{12} a_{23} -2 a_{2} a_{3} a_{12} a_{13} + 4 a_{\emptyset} a_{23} a_{13} a_{12}+ 4 a_{123} a_{1} a_{2} a_{3}.
 \end{align*}
	This quartic polynomial therefore appears in the arithmetic conditions on the image of the principal minor map. 
	
	\begin{thm}\label{thm:HermImage} 
		Let ${\bf a} = (a_S)_{S\subseteq [n]}\in \F^{2^n}$ with $a_{\emptyset}=1$. There exists a Hermitian matrix over $\K$ with principal minors ${\bf a}$ if and only if for every $\gamma\in \SL_2(\F)^n \rtimes S_n$:
		\begin{itemize}
			\item[(i)]  $\gamma \cdot (a_1a_2-a_{\emptyset}a_{12})$ is a Hermitian square in $\K$,
			%\item[(ii)] $\frac{1}{\delta^2} {\rm Disc}_{x_3}(\gamma \cdot \Delta_{12}(f_{\bf a}))$ is a square.
			\item[(ii)] $\frac{1}{\delta^2}\HypDet({\bf \gamma \cdot a})$ is a square in $\F$, and 
			\item[(iii)] $\gamma\cdot {\bf a}$ satisfies the degree-12 polynomials given by specializing 
			$B_{x_4}\left({\rm Discr}_{x_3}(\gamma\cdot \Delta_{12}f_{\bf a})\right)$, $C_{x_4}\left({\rm Discr}_{x_3}(\gamma\cdot \Delta_{12}f_{\bf a})
			\right)$
			and $D_{x_4}\left({\rm Discr}_{x_3}(\gamma\cdot \Delta_{12}f_{\bf a})\right)$ to $x_5 = \hdots = x_n =0$. 
		\end{itemize}
	\end{thm}
Here the operators $B_{x}$, $C_x$, $D_x$ are defined in Corollary~\ref{cor:quarticSquare}. 
	\begin{proof}[Proof of Theorem~\ref{thm:HermImage}]
			By Theorem~\ref{thm:HerDetRep} with $n=m$, ${\bf a} = (a_S)_{S\subseteq [n]}\in \F^{2^n}$ is in the image of the principal minor map if and only if $\Delta_{ij}(f_{{\bf a}})$ is a Hermitian square for all $i,j\in[n]$, which according to Theorem~\ref{thm:conj}, is satisfied if and only if for all $\gamma \in \SL_2(\F)^n\rtimes S_n$, $\gamma \cdot \Delta_{34}(f_{{\bf a}})|_{x_5 = \hdots = x_n =0}$ is a Hermitian square in $\K[x_1,x_2]$. This is equivalent to the three hypothesis of Theorem~\ref{thm:Conj}, which in turn is equivalent to the three hypotheses of the theorem.
	\end{proof}
		
		Taking $\K=\C$ with the action complex conjugation then gives the following. 
		
	\begin{cor}\label{cor:ImgHer}
		Let ${\bf a} = (a_S)_{S\subseteq [n]}\in \R^{2^n}$ with $a_{\emptyset}=1$. There exists a Hermitian matrix over $\C$ with principal minors ${\bf a}$ if and only if for every $\gamma\in \SL_2(\R)^n \rtimes S_n$
		\begin{itemize}
			\item[(ii)] $\gamma\cdot (a_1a_2-a_{\emptyset}a_{12})\geq 0$, 
			\item[(ii)] $\HypDet(\gamma \cdot{\bf a})\leq 0$, and 
			\item[(ii)] $\gamma\cdot {\bf a}$ satisfies the three degree-12 equations given by restricting 
			$B_{x_4}\left({\rm Discr}_{x_3}(\Delta_{12}f_{\gamma\cdot{\bf a}})\right)$, 
			$C_{x_4}\left({\rm Discr}_{x_3}(\Delta_{12}f_{\gamma\cdot{\bf a}})
			\right)$
			and $D_{x_4}\left({\rm Discr}_{x_3}(\Delta_{12}f_{\gamma\cdot{\bf a}})\right)$ to $x_5 = \hdots = x_n =0$.
		\end{itemize}
	\end{cor}	
		
%%%%%%%%%%%%%%%%%%%%%%%%%%%%%%%%%%%%%%%%%%%%%%%%
%%%%%%%%%%%%%%%%%%%%%%%%%%%%%%%%%%%%%%%%%%%%%%%%%

\section{A family of counterexamples}\label{sec:counterex}

Let $\F$ be a field and for $n\geq 2$, consider the multiaffine polynomial $f_{2n+1}\in \F[x_1, \hdots, x_{2n+1}]$ given by 
\begin{equation}\label{eq:f2n+1}
f_{2n+1} = x_1\cdot \prod_{j=1}^n (x_{2j+1}x_{2j+2} + 1) 
+ \prod_{j=1}^n (x_{2j}x_{2j+1} + 1)
\end{equation}
where we take $x_{2n+2} = x_2$. 
We show that this polynomial is not determinantal, i.e. 
its vector of coefficients do not belong to the image of the principal minor map, but is determinantal after specializing any one variable:
\begin{thm}\label{thm:GenCase}
There is \emph{no finite set of equations} 
whose orbit under ${\rm SL}_2(\F)^n\rtimes S_n$ 
set-theoretically cuts out the image of the principal minor map  for all $n$.
\end{thm} 

Let $I_n \subset \F[a_S : S\subseteq [n]]$ be the homogeneous ideal of polynomials 
vanishing on the image of $n\times n$ matrices under the principal minor map in $\mathbb{P}^{2^n-1}(\F)$. 
There is a natural inclusion of $I_n$ into  $\F[a_S : S\subseteq [n+1]]$. 

\begin{thm}\label{thm:CounterEx}
The coefficient vector of the polynomial $f_{2n+1}$ belongs to the variety of polynomials in the orbit 
$({\rm SL}_2(\F)^{2n+1} \rtimes S_{2n+1})\cdot I_{2n}$ 
but not the variety of $I_{2n+1}$. 
\end{thm}

The proof of this theorem relies on the fact that the coefficient of any generic specialization of $f_{2n+1}$ lies in the image 
of the principal minor map, up to scaling.  
One key observation is that the Rayleigh differences of $f_{2n+1}$ 
do not all factor as the product of \emph{two} multiaffine polynomials, 
but do have such factorizations after specializing anyone variable.  We show this explicitly by writing down the determinantal representations of these specializations. 

\begin{lem}\label{lem:Det_x1}
The rational function $\frac{1}{x_{2n+1}}f_{2n+1}$ can be written as 
$\det({\rm diag}(x_1, \hdots, x_{2n}) +B)$ where for $1\leq i,j\leq 2n$,

The rational function $\frac{1}{1+x_1}f_{2n+1}$ can be written as 
$\det({\rm diag}(x_2, \hdots, x_{2n+1}) +A)$ where for $2\leq i,j\leq 2n+1$, 
\[
A_{ij} = \begin{cases} 
1/(1+x_1) &  \text{ if $i$ is odd, $j$ is even, and  $i>j$} ,\\
-x_1/(1+x_1) &  \text{ if $i$ is odd, $j$ is even, and  $i<j$}, \\
-1  &  \text{ if $i$ is even, $j = i+1$}, \\
1  &  \text{ if $i$ is even, $j = i-1$}, \\
-x_1 &  \text{ if $i=2$, $j = 2n+1$, and } \\
0 & \text{ otherwise}.
\end{cases}
\]
\end{lem}
\begin{proof}
Let $D$ denote the determinant of the matrix $M = \det({\rm diag}(x_2, \hdots, x_{2n+1}) +A)$. 
By definition, $D$ is a polynomial in $\frac{1}{1+x_1}, x_1, x_2, \hdots, x_n$. 
Moreover the entries for which $x_1+1$ appears in the denominator 
form a square submatrix whose rows correspond to odd indices and whose columns correspond to even ones. 
It has the form 
\[
\frac{1}{1+x_1}
\begin{pmatrix}
1 & -x_1 	& -x_1 & \hdots & -x_1 \\
1 &  1 	& -x_1 & \hdots & -x_1 \\
\vdots& \ddots   &\ddots  & \ddots   &\vdots   \\
1 & 1 & \ddots  & \ddots & -x_1 \\
1 & 1 & 1 & \hdots &1 
\end{pmatrix}
= 
\frac{1}{1+x_1}J - U
\]
where $J$ is the all ones matrix and $U$ is an upper triangular matrix with $U_{ij} = 1$ for $i<j$ and $U_{ij}=0$ otherwise. Since $J$ has rank one, the exponent of $1+x_1$ appearing in the denominator 
of \emph{any} minor of this matrix is at most one. 
This also shows that despite the many appearances of $x_1$ in numerator of this matrix, 
it does not appear in the numerator of any minor.  There is only one other entry in $M$ containing $x_1$, 
and so the determinant $D$ can be written as $(x_1+1) ^{-1} p_1 + p_2$ where $p_1$ and $p_2$ are multiaffine in $x_1, \hdots, x_{2n+1}$. 
Moreover, the only term in the Laplace expansion of the determinant of $M$ avoiding this submatrix is the 
product of the diagonal $\prod_{j=2}^{2n+1}x_j$. Therefore we can write $D$ as $(x_1+1) ^{-1} p$
where $p$ is multiaffine in $x_1, \hdots, x_{2n+1}$. 
Therefore to show that $p = f_{2n+1}$ it suffices to show that they have the same specialization at $x_1=0$ 
and the same coefficient of $x_1$.

When we specialize $x_1$ to zero, $M$ becomes a block upper-triangular l matrix with diagonal blocks of the form $\begin{pmatrix} x_{2j} & -1 \\ 1 & x_{2j+1}\end{pmatrix}$. 
Its determinant agrees with the specialization of $\frac{1}{1+x_1}f_{2n+1}$ to $x_1 = 0$.

Consider the rational function $g$ obtained by inverting $x_1$ in $\frac{1}{1+x_1}f_{2n+1}$, which is
\[\frac{x_1}{1+x_1}f_{2n+1}(x_1^{-1}, x_2, \hdots, x_n)
= \frac{1}{1+x_1} \cdot  
\left(
\prod_{j=1}^n (x_{2j+1}x_{2j+2} + 1) 
+ x_1\cdot  \prod_{j=1}^n (x_{2j}x_{2j+1} + 1)
\right).
\]
Let $M'$ be the matrix obtained from $M$ by replacing $x_1$ by $x_1^{-1}$ and then multiplying the 
column indexed by $2$ by $x_1^{-1}$ and the row indexed by $2$ by $x_1$.
The entries are now rational functions in $x_1$ with only $1+x_1$ appearing in the denominator. 
 After specializing $M'$ to $x_1=0$ and cyclic shifting the rows and columns by one, 
 we find another block upper triangular matrix with diagonal blocks of the form 
 $\begin{pmatrix} x_{2j+1} & -1 \\ 1 & x_{2j+2}\end{pmatrix}$ for $j=1, \hdots, n-1$ and 
 $\begin{pmatrix} x_{2n+1} & 1 \\ -1 & x_{2}\end{pmatrix}$.
  Therefore the determinant of $M'$ restricted to $x_1 = 0$ is given by $\prod_{j=1}^n (x_{2j+1}x_{2j+2} + 1)$.

By definition, the determinant of $M'$ equals 
$D(x_1^{-1}, x_2,\hdots, x_n) =  \frac{x_1}{1+x_1}p(x_1^{-1}, \hdots, x_n) $. 
Restricting to $x_1=0$ gives the coefficient of $x_1$ in $p$, which must be 
$\prod_{j=1}^n (x_{2j+1}x_{2j+2} + 1)$.   
Therefore $p$ agrees with the polynomial $f_{2n+1}$. 
\end{proof}

\begin{figure}
\[
{\small \left(
\begin{array}{cccccc}
 x_2 & -1 & 0 & 0 & 0 & -x_1 \\
 \frac{1}{x_1+1} & x_3 & -\frac{x_1}{x_1+1} & 0 & -\frac{x_1}{x_1+1} & 0 \\
 0 & 1 & x_4 & -1 & 0 & 0 \\
 \frac{1}{x_1+1} & 0 & \frac{1}{x_1+1} & x_5 & -\frac{x_1}{x_1+1} & 0 \\
 0 & 0 & 0 & 1 & x_6 & -1 \\
 \frac{1}{x_1+1} & 0 & \frac{1}{x_1+1} & 0 & \frac{1}{x_1+1} & x_7 \\
\end{array}
\right)
\sim 
\left(
\begin{array}{cccccc}
 x_2 & 0 & 0 & -1 & 0 & -x_1 \\
 0 & x_4 & 0 & 1 & -1 & 0 \\
 0 & 0 & x_6 & 0 & 1 & -1 \\
 \frac{1}{x_1+1} & -\frac{x_1}{x_1+1} & -\frac{x_1}{x_1+1} & x_3 & 0 & 0 \\
 \frac{1}{x_1+1} & \frac{1}{x_1+1} & -\frac{x_1}{x_1+1} & 0 & x_5 & 0 \\
 \frac{1}{x_1+1} & \frac{1}{x_1+1} & \frac{1}{x_1+1} & 0 & 0 & x_7 \\
\end{array}
\right)}
\]
\caption{The matrix $A$ in Lemma~\ref{lem:Det_x1} for $2n+1 = 7$.}
\end{figure}

\begin{lem}\label{lem:Det_x2n+1}
For every $m=2, \hdots, 2n+1$, the coefficients of 
$\frac{1}{x_{m}}f_{2n+1}$ are the principal minors of a $2n\times 2n$ matrix 
with entries in $\{0,\pm 1, x_m^{\pm1}\}$. 
In particular, the rational function $\frac{1}{x_{2n+1}}f_{2n+1}$ can be written as 
$\det({\rm diag}(x_1, \hdots, x_{2n}) +B)$ where for $1\leq i,j\leq 2n$, 
\[
B_{ij} = \begin{cases} 
1 &  \text{ if $j=i+1$ and $i>1$ or $(i,j) =(1,1)$ or $(i,j) =  (2,1)$} ,\\
-1  &  \text{ if $i$ is even and $j = i-1$ or $(i,j) = (2n,1)$}, \\
x_{2n+1} &  \text{ if $i$ odd, $i\geq 3$, and $j = 1$, } \\
1/x_{2n+1}&  \text{ if $i\in \{1,2\}$ and $j$ is even, and} \\
0 & \text{ otherwise}.
\end{cases}
\]
\end{lem}
\begin{proof}
Let $M = \det({\rm diag}(x_1, \hdots, x_{2n}) +B)$ and let $D$ denote its determinant. 
As in the proof of Lemma~\ref{lem:Det_x1}, the entries of $M$ with $x_{2n+1}$ appearing in the denominator 
appear in a submatrix of rank-one. The entries with $x_{2n+1}$ appearing in the numerator 
are contained in the first column.  Moreover, in the Laplace expansion of the determinant, 
the only terms avoiding the submatrix of entries $x_{2n+1}^{-1}$ must include the $(1,1)$ and $(2,2)$ entries, and 
so will not involve any entries with $x_{2n+1}$. 
It follows that $D$ can be written as $x_{2n+1}^{-1}p$ where $p$ is multiaffine in $x_1, \hdots, x_{2n+1}$. 
Therefore it suffices to check that $f_{2n+1}$ and $p$ have the same restriction to $x_{1} = 0$ 
and same coefficient of $x_{1}$.

We see that the coefficient of $x_1$ in $D$ is the determinant of the matrix $M$ after removing the 
first row and column. This minor is a block matrix with one block of the form $(x_2 + 1/x_{2n+1})$ 
and the rest of the form $\begin{pmatrix}x_{2j+1} & 1 \\ -1 & x_{j+2}\end{pmatrix}$. 
Therefore the coefficient of $x_1$ in $p$ and $f_{2n+1}$ agree.

The specialization of $M$  to $x_1 =0$ is a matrix has the form $\begin{pmatrix} 1 & b^T \\ 
c & A \end{pmatrix}$. Using Schur complements, we see that the determinant 
equals the determinant of $A - cb^T$.  One can check that the matrix $A - cb^T$ is a block-lower triangular matrix
with diagonal blocks  $\begin{pmatrix} x_{2j} & 1 \\ -1 & x_{2j+1}\end{pmatrix}$ for $j=1, \hdots, n-1$ 
and $x_{2n} + 1/x_{2n+1}$. This shows that the restriction of $p$ to $x_1=0$ agrees with that of $f_{2n+1}$. 

For the corresponding statement with arbitrary $m\neq 1$, we use the symmetries of $f_{2n+1}$ under the 
action of a dihedral group of order $n$ with the cyclic action  $j\mapsto j+2$  (identifying $2n+j = j$ for $j\geq 2$ 
and reflection $n+1-j \leftrightarrow n+2-j$. 
There is some element of this group that moves $m$ to $2n+1$, and we can take the 
image of the representation above. 
\end{proof}

\begin{figure}
\[
{\small \left(
\begin{array}{cccccc}
 x_1+1 & \frac{1}{x_7} & 0 & \frac{1}{x_7} & 0 & \frac{1}{x_7} \\
 1 & x_2+\frac{1}{x_7} & 1 & \frac{1}{x_7} & 0 & \frac{1}{x_7} \\
 x_7 & 0 & x_3 & 1 & 0 & 1 \\
 0 & 0 & -1 & x_4 & 1 & 0 \\
 x_7 & 0 & 0 & 0 & x_5 & 1 \\
 -1 & 0 & 0 & 0 & -1 & x_6 \\
\end{array}
\right)} \ \ \ \ \ 
{\small \left(
\begin{array}{ccccc}
 x_2 & 1 & 0 & 0 & 0 \\
 -1 & x_3 & 0 & 0 & 0 \\
 0 & -1 & x_4 & 1 & 0 \\
 -1 & 0 & -1 & x_5 & 0 \\
 \frac{1}{x_7} & 0 & \frac{1}{x_7} & -1 & x_6+\frac{1}{x_7} \\
\end{array}
\right)}
\]
\caption{The matrices $B$ (left) and $A - cb^T$ (right) in Lemma~\ref{lem:Det_x2n+1} for $2n+1 = 7$.}
\end{figure}

To show that $f_{2n+1}$ does not belong to $I_{2n+1}$, we will use the following: 

\begin{lem}\label{lem:ZariskiClosure}
The set of polynomials 
\[\mathcal{F}_n = \left\{f\in \F[{\bf x}]_{\rm MA} : \text{ for all } i,j\in [n], \Delta_{ij}(f) = 
g_{ij}\cdot h_{ij} \text{ for some } g_{ij},h_{ij}\in \F^{\rm alg}[{\bf x}]_{\rm MA}\right\}\]
is Zariski closed in $\F[{\bf x}]_{\rm MA} \cong \F^{2^{[n]}}$, where $\F^{\rm alg}$ denotes the algebraic closure of $\F$. 
\end{lem}
\begin{proof} The set of multiquadratic polynomials in $ \F^{\rm alg}[{\bf x}]_{\rm MQ}$ that factor as the product of two multiaffine polynomials is the image of $ \F^{\rm alg}[{\bf x}]_{\rm MA} \times  \F^{\rm alg}[{\bf x}]_{\rm MA}$ under 
$(g, h)\mapsto g\cdot h$.  Since this map is bilinear, it follows from 
the projective elimination theorem that the set $\{q\in \F^{\rm alg}[{\bf x}]_{\rm MQ} : q = g\cdot h \text{ for some }g,h\in \F^{\rm alg}[{\bf x}]_{\rm MA}\}$ is Zariski-closed in $\F^{\rm alg}[{\bf x}]_{\rm MQ}$. 

Pulling back by the map $\Delta_{ij}$, it follows that for each $i,j\in [n]$, 
the set of polynomials $f\in \F^{\rm alg}[{\bf x}]_{\rm MA}$ for which $\Delta_{ij}(f)$ factors as the product of two multiaffine polynomials is Zariski-closed, as is their intersection over all $i,j\in [n]$. It follows that its intersection with $\F[{\bf x}]_{\rm MA}$ is Zariski-closed $\F[{\bf x}]_{\rm MA}$.
\end{proof}

Theorem~\ref{thm:DeltaDetRep} implies that the image of $\F^{n\times n}$ under the principal minor map is a subset of the variety $\mathcal{F}_n$, although as Example~\ref{ex:CounterExFactoring} shows, this containment can be strict. In order to show that $f_{2n+1}$ does not belong to the variety of 
$I_{2n+1}$, it suffices to show that $f_{2n+1}$ does not belong to $\mathcal{F}_{2n+1}$. 

Recall that for $f = \sum_{S\subseteq [n]} a_{S} {\bf x}^{[n]\backslash S}$, the coefficient vector of $f$ is defined to be 
\[{\rm coeff}(f) = (a_S)_{S\subseteq [n]} \in \F^{2^{[n]}}.\]

\begin{proof}[Proof of Theorem~\ref{thm:CounterEx}]
For convenience, let $f = f_{2n+1}$. 
Let $P\in I_{2n}$ be a homogenous polynomial vanishing on the image of $\F^{2n\times 2n}$ 
under the principal minor map. 
Let $Q$ denote the image of $P$ under inclusion into $\F[a_S:S\subseteq [2n+1]]$. 
Note that $Q$ only involves $a_{S}$ with $2n+1\not\in S$. 
Since our indexing of coefficients is inclusion reversing, we see that
the evaluation of $Q$ at the coefficient vector of $f$ depends only 
on coefficients of monomials containing $x_{2n+1}$. 
In particular, its evaluation at the coefficient vector of $f$ equals 
the evaluation of $P$ at the coefficient vector of derivative of $f$ with respect to $x_{2n+1}$, i.e. 
\begin{equation}\label{eq:P_no2n+1}
Q({\rm coeff}(f)) = P({\rm coeff}(\partial f/\partial x_{2n+1} )).   
\end{equation}

If $\F$ is finite, it suffices to replace it with any infinite field extension, such as $\F(t)$ or $\F^{\rm alg}$.  
Let $(\gamma, \pi)\in {\rm SL}_2(\F)^{2n+1} \rtimes S_{2n+1}$, with $\gamma$ generic, where $\F^{\rm alg}$
denote the (necessarily infinite) algebraic closure of $\F$. 
We can write $(\gamma, \pi)$ as the composition of elements 
$(\widehat{\gamma}, \widehat{\pi})$ in $ {\rm SL}_2(\F)^{2n} \rtimes S_{2n}$
and $(\gamma_{2n+1}, \sigma)$, where 
$\gamma_{2n+1} = {\tiny \begin{pmatrix} a & b \\ c & d \end{pmatrix}}\in  {\rm SL}_2(\F)$ acts on $x_{2n+1}$ and $\sigma$ is the transposition $\sigma = (m (2n+1))\in S_{2n+1}$.
Then 
\[
(\gamma_{2n+1}, \sigma)\cdot f = 
(c x_{2n+1} + d)f\left(x_1, \hdots, x_{m-1}, \frac{ax_{2n+1} + b}{c x_{2n+1} + d}, x_{m+1}, \hdots, x_{2n}, x_{m}\right).
\]
By the genericity of $\gamma$, $c\neq 0$ and 
\[
\frac{\partial}{\partial x_{2n+1}}\!\left((\gamma_{2n+1}, \sigma)\cdot f \right) 
= c f\bigl|_{\displaystyle{\{x_m = a/c, x_{2n+1} = x_m\}}}.
\]
Call this polynomial $g$. The coefficient $\lambda$ of $\prod_{i=1}^{2n}x_i$ in $g$
is $a+c$ for $m=1$ and $a$ for $m>1$. In either case, we can assume it is nonzero by the genericity of $\gamma$. 

By Lemma~\ref{lem:Det_x1} for $m=1$ and Lemma~\ref{lem:Det_x2n+1} for $m>1$, 
the polynomial $\frac{1}{\lambda} g$  is determinantal and its coefficient vector belongs to the image of the principal minor map. 
Since the image of the principal minor map is invariant under 
the action of $({\rm SL}_2(\F)^{2n} \rtimes S_{2n})$,
by \eqref{eq:P_no2n+1}, 
 \[
 0 = P({\rm coeff}(g)) = P({\rm coeff}((\widehat{\gamma}, \widehat{\pi})\cdot g))
 =  Q({\rm coeff}((\gamma,\pi)\cdot f)).
 \]
This shows that the coefficient vector of $f$ belongs to the variety of 
$({\rm SL}_2(\F)^{2n+1} \rtimes S_{2n+1})\cdot I_{2n}$. 

On the other hand, 
we calculate that 
\[
\Delta_{12}(f) = 
 (x_3 - x_{2n+1} ) 
\prod_{i=3}^{2n}(x_{i}x_{i+1} + 1).
\]
These form a cycle of length $2n-1$ of irreducible bivariate factors,  which cannot 
be factored as the product of two multiaffine polynomials. 
It follows that $f$ does not belong to the variety $\mathcal{F}_{2n+1}$ from Lemma~\ref{lem:ZariskiClosure}, 
which, by Theorem~\ref{thm:DeltaDetRep}, 
contains the variety of $I_{2n+1}$. 
\end{proof}

The polynomial $f_{2n+1}$ shows that the orbit of the ideal $I_{2n}$ under $({\rm SL}_2(\F)^{2n+1} \rtimes S_{2n+1})$ is not 
enough to cut out the set of polynomials $f\in \F[x_1, \hdots, x_{2n+1}]$ all of whose Rayleigh differences factor as the product of two multiaffine polynomials.  As Example~\ref{ex:CounterExFactoring} shows, even this is not enough to cut out the image of the principal minor map.

\bibliographystyle{plain}

\end{document}